\documentclass[a4paper,11pt]{article}
\usepackage{amssymb}
\usepackage{amsfonts}
\usepackage{amsmath}
\usepackage{graphicx}
\usepackage[all]{xy}
\usepackage{array}

\usepackage{amsthm}

\newtheorem{theorem}{Theorem}[section]

\newtheorem{lemma}[theorem]{Lemma}
\newtheorem{proposition}[theorem]{Proposition}
\newtheorem{corollary}[theorem]{Corollary}

\newtheorem*{problemnn}{Problem}

\theoremstyle{definition}
\newtheorem{definition}[theorem]{Definition}
\newtheorem{condition}[theorem]{Condition}

\newtheorem{notation}[theorem]{Notation}

\newtheorem{remark}[theorem]{Remark}

\newtheorem{example}[theorem]{Example}

\theoremstyle{remark}

\newcommand{\Si}{\mathfrak{S}}
\newcommand{\Ext}{{\mathrm{Ext}}}

\renewcommand{\hom}{{\mathrm{Hom}}}

\newcommand{\End}{{\mathrm{End}}}
\newcommand{\Id}{{\mathrm{Id}}}

\newcommand{\PP}{\mathcal{P}}

\newcommand{\kk}{\Bbbk}
\newcommand{\Z}{\mathbb{Z}}
\newcommand{\C}{\mathcal{C}}

\newcommand{\DD}{\mathbf{D}}

\newcommand{\Soc}{\mathrm{Soc}}
\newcommand{\Head}{\mathrm{Head}}

\newcommand{\St}{\mathcal{S}t}
\newcommand{\StBif}{\mathcal{S}t'(d,e,r)}
\newcommand{\uotimes}{\underline{\otimes}}

\newcommand{\uHom}{\underline{\mathrm{Hom}}}
\input cyracc.def

\begin{document}

\title{Connectedness of cup products for polynomial representations of $GL_n$ and applications}

\author{Antoine Touz\'e}
\date{\today}

\maketitle

\begin{abstract}
We find conditions such that cup products induce isomorphisms in low degrees for extensions between stable polynomial representations of the general linear group. We apply this result to prove generalizations and variants of the Steinberg tensor product theorem. Our connectedness bounds for cup product maps depend on numerical invariants which seem also relevant to other problems, for example to study the cohomological behavior of the Schur functor.
\end{abstract}

\section{Introduction}

Let $\kk$ be a field of positive characteristic $p$, and let $G$ be an algebraic group over $\kk$. The category of rational representations of $G$ (as in \cite{Jantzen}) is equipped with a tensor product, which induces a cup product on extension groups :
$$\Ext^*_G(M,N)\otimes \Ext^*_G(P,Q)\xrightarrow[]{\cup} \Ext^*_G(M\otimes P,N\otimes Q)\;.$$
Of course the cup product is injective (but usually not surjective) in cohomological degree zero, and in general it is neither injective nor surjective in higher degrees.
If $G=GL_n(\kk)$, it was observed in \cite{TouzeClassical} that the cup product is injective in all degrees when $M$, $N$, $P$, $Q$ are \emph{stable polynomial representations}, i.e. when $M$, $N$, $P$, $Q$  are polynomial representations in the usual sense \cite{Green,Martin} and furthermore when $n$ is big enough with respect to their degrees. This surprising fact is easily proved by using the description of stable polynomial representations in terms of the strict polynomial functors of Friedlander and Suslin \cite{FS}. 

The first main result of this article is theorem \ref{thm-1}, which establishes conditions under which cup products are not only injective, but also surjective in low degrees. Theorem \ref{thm-1} actually applies to the case where $N$ and $Q$ are representations twisted $r$-times by the Frobenius morphism, i.e. for cup products of the form:
$$\Ext^*_G(M,N)\otimes \Ext^*_G(P^{(r)},Q^{(r)})\xrightarrow[]{\cup} \Ext^*_G(M\otimes P^{(r)},N\otimes Q^{(r)})\;.$$
As for injectivity in \cite{TouzeClassical}, the natural home for stating and proving this connectedness property of cup products is the category of strict polynomial functors. We note that already in degrees $0$ and $1$, our theorem looks much stronger than what was previously known for the behaviour of cup products, see remark \ref{rk-Jantzen}.

We then give concrete applications of theorem \ref{thm-1} to the representation theory of $GL_n(\kk)$. Namely we prove the following two new generalizations of Steinberg tensor product theorem.
\begin{itemize}
\item We call \emph{tensor products of Steinberg type} the stable polynomial representations of the form $M\otimes N^{(r)}$, where all the composition factors of $M$ have $p^r$-restricted highest weights. Representations of this form appear naturally, e.g. in the theory of good $p$-filtrations \cite{Andersen}.

In theorem \ref{thm-eq}, we describe the structure of the abelian subcategory generated by these tensor products of Steinberg type (with $r$ and $\deg M$ fixed). In particular we prove that the $GL_n(\kk)$-module $M\otimes N^{(r)}$ has the same structure as the $GL_n(\kk)\times GL_n(\kk)$-module $M\otimes N$. This is interesting because the latter is much easier to study. (The classical Steinberg tensor product theorem corresponds to the very special case where $M$ and $N$ are simple. Indeed in that case $M\otimes N$ is simple as a $GL_n(\kk)\times GL_n(\kk)$-module, thus by our theorem the $GL_n(\kk)$-module $M\otimes N^{(r)}$ is simple too).  
\item As made explicit in \cite{Krause}, stable polynomial representations are equipped with an internal tensor product (Day convolution product), dual to the internal Hom used in $\Ext$-computations e.g. in \cite{TouzeRingel,TouzeEML}. In theorem \ref{thm-SteinInternal} we explain how to reduce the computation of internal tensor products of simples to the case of simples with $p$-restricted highest weights. Thus, theorem \ref{thm-SteinInternal} plays the same role for understanding internal tensor products of simples as the classical Steinberg tensor product theorem does for understanding ordinary tensor products of simples.
\end{itemize}
In appendix \ref{app-Stein} we show that theorem \ref{thm-1} can also be used to derive new proofs of two well-known fundamental theorems for simple representations of $GL_n(\kk)$, namely Steinberg tensor product theorem and Clausen and James' theorem. We note that another functor technology proof of Steinberg tensor product theorem is given in \cite{Kuhn}. The proof given here seems quite different, see remark \ref{rk-Kuhn}.

\medskip

The bounds for connectedness given in theorem \ref{thm-1} depend on some cohomological constants $p(M,r)$ and $i(N,r)$. To be more specific, a projective stable polynomial module is $p^r$-bounded if its socle is a direct sum of simples with $p^r$-restricted highest weights, cf. corollary \ref{cor-bounded}.
The integer $p(M,r)$ is characterized by: 

\begin{center} 
$p(M,r)\ge k$ if and only if there exists a resolution $P$ of $M$ in which the first $k$ terms $P_0$,\dots, $P_{k-1}$ are $p^r$-bounded projectives.
\end{center}

The integer $i(N,r)$ is defined dually, cf definition \ref{def-pi}. Although we use this definition for stable polynomial representations, it makes sense for unstable polynomial representations as well. We are not aware of previous occurrences of these constants in the literature. We study their basic properties and give characterizations of these constants, as well as elementary computation rules and examples. In most examples, we limit ourself to giving estimates for these constants rather than exact values, and leave the following challenging problem open.

\begin{problemnn}
Compute (or obtain a reasonnable understanding of) the exact value of $p(M,r)$ and $i(M,r)$ for the most important $GL_n(\kk)$-modules (simple modules, standard or costandard modules).
\end{problemnn}

One further motivation for this problem is that the constants $p(M,r)$ and $i(M,r)$ seem to be related to other problems of interest. Let us give two examples. 
\begin{itemize}
\item In theorem \ref{thm-KN}, we prove that the constants $p(M,1)$ and $i(M,1)$ govern the connectedness of the Schur functor on the level of extensions. The cohomological behavior of the Schur functor was already studied in a series of papers \cite{DEN,KN,KS}. Our theorem \ref{thm-KN} gives a simpler and effective approach to this problem. For example, with our first computations of $i(F,1)$ and $p(F,1)$, we recover and generalize many results from Kleshchev and Nakano in \cite{KN}.

\item It seems that the values of $p(L,r)$ capture some interesting concrete properties of simple functors $L$. Clausen and James' theorem \cite{Clausen, James} says that $p(L,1)>0$ if and only if the highest weight of $L$ is $p$-restricted. Reischuk has proved \cite{Reischuk} that $p(L,1)>1$ if and only if $Q^d\uotimes L$ is simple, where $Q^d$ is the simple head of $S^d$ (see section \ref{sec-tensint} and in particular theorem \ref{thm-SteinInternal} and corollaries \ref{cor-intss} and \ref{cor-lesdeuxpareil} to understand why this particular internal tensor product is interesting). It would be interesting to understand if higher inequalities $p(L,1)>k$ (of cohomological nature) are directly connected to some non-homological representation theoretic properties of $L$.
\end{itemize}

We finish by explaining a wider perpective behind the work presented in this article. Functor category techniques have proved useful for studying representations and cohomology of many variants of classical matrix groups. See for example \cite{TouzeClassical} for symplectic and orthogonal group schemes, \cite{Axtell,Drupieski} for super Schur algebras, \cite{HY} for quantum $GL_n$, \cite{FFSS,DV} for finite classical groups or more generally \cite{Djament} for discrete unitary groups. In all these examples, the functor categories in play share many properties with the category of strict polynomial functors used here. So we expect that the techniques and results developed in this article can be adapted to these cases. For example, we prove in \cite{TouzeICRA} an analogue of theorem \ref{thm-1} for polynomial representations of othogonal and symplectic group schemes.

\medskip

This article has been written in such a way that the main thread of ideas and proofs is self-contained. In particular, only very basic facts of the representation theory of general linear groups are used (the highest weight category structure is used only for the results of section \ref{subsec-ex}) and no combinatorics of the symmetric group is used (except a result of Bessenrodt and Kleshchev \cite{BK} in corollary \ref{cor-intss}). These basic facts are recalled in section \ref{sec-back}. In the same spirit, we have also added an appendix on representations of tensor products of finite dimensional algebras, whose results are used in section \ref{sec-tensS}.

\subsection*{Acknowledgments}
We are grateful to Lionel Schwartz for communicating to us an alternative proof of theorem \ref{thm-tensprespres} relying on Steenrod operations. We are also grateful to Alexander Zimmermann for several discussions on representations of finite dimensional algebras, and to Shraddha Srivastava for pointing out a missing hypothesis in a first version of proposition \ref{prop-calcul}.

\section{Background}\label{sec-back}
\subsection{Strict polynomial functors and Schur algebras}\label{subsec-recallbasic}
In this article $\kk$ is a field of positive characteristic $p$, and $\PP_{d,\kk}$ denotes the category of homogeneous strict polynomial functors of degree $d$ over $\kk$ (with possibly infinite dimensional values). We refer e.g. to \cite{Panorama}, \cite{FS} or \cite{Krause} for an introduction to strict polynomial functors. 
If one considers an infinite ground field $\kk$, strict polynomial functors have a nice description like the one in \cite{MacDonald} (where they are simply called `polynomial functors'). Namely, strict polynomial functors are functors from finite dimensional $\kk$-vector spaces to $\kk$-vector spaces, such that for all pairs of finite dimensional vector spaces $(V,W)$, the map 
$$\begin{array}{ccc} \hom_\kk(V,W)& \to & \hom_\kk(F(V),F(W))\\
f& \mapsto & F(f)   
\end{array}
$$
is given by a homogeneous polynomial of degree $d$ (that is, given by an element of $S^d(\hom_\kk(V,W)^*)\otimes \hom_\kk(F(V),F(W))$). 

For example, the category $\PP_{0,\kk}$ is equivalent to the category of constant functors, which is equivalent to the category of $\kk$-vector spaces. Typical examples of homogeneous functors of higher degree $d$ are the tensor product functors $\otimes^d:V\mapsto V^{\otimes d}$, the symmetric power functors $S^d:V\mapsto (V^{\otimes d})_{\Si_d}$ and the divided power functors $\Gamma^d:V\mapsto (V^{\otimes d})^{\Si_d}$. (Here the symmetric group $\Si_d$ acts on $V^{\otimes d}$ by permuting the factors of the tensor product). Note that $S^0=\otimes^0=\Gamma^0=\kk$, $S^1=\otimes^1=\Gamma^1$ but for $d\ge p$ the functor $S^d$ is not isomorphic to $\Gamma^d$.

We denote by $\PP_\kk$ the category of strict polynomial functors (of bounded degree), that is $\PP_\kk=\bigoplus_{d\ge 0}\PP_{d,\kk}$. By evaluating a strict polynomial functor $F$ on $\kk^n$, one obtains a polynomial $GL_n(\kk)$-module $F(\kk^n)$. Restricting to homogeneous strict polynomial functors of degree $d$, one obtains a functor
$$\PP_{d,\kk}\xrightarrow[]{\mathrm{ev}_{\kk^n}} \mathrm{Pol}_{d,GL_n(\kk)}\simeq S(n,d)\text{-Mod}\;,$$
where $\mathrm{Pol}_{d,GL_n(\kk)}$ denotes the category of homogeneous polynomial representations of $GL_n(\kk)$ of degree $d$, and  $S(n,d)\text{-Mod}$ the equivalent category of modules over the Schur algebra $S(n,d)$ (which is finite dimensional).
It is an equivalence of categories, provided $n\ge d$.
In particular $\PP_\kk$ has nice properties similar to the ones of modules over a finite dimensional algebra. We shall use the following ones in the sequel.
\begin{enumerate}
\item Simple functors are homogeneous functors, and their values are finite dimensional vector spaces. A functor has a finite composition series if and only if it has finite dimensional values; such functors are called \emph{finite}. Finally, every functor is the union of its finite subfunctors.
\item Arbitrary direct sums of injective functors are injective, and every functor can be embedded into a direct sum of finite injectives.
\item Any nonzero homogeneous functor has a nonzero socle, a nonzero head and a finite Loewy length.
\end{enumerate}

\subsection{Elementary facts regarding simple functors}\label{subsec-simple}

Simple polynomial $GL_n(\kk)$-modules are traditionnaly classified by examining the action of a maximal torus on $GL_n(\kk)$-modules, that is using the concept of highest weights, see e.g. \cite[Chap. 1]{Martin}. In the sequel of the article, we shall use the following consequences of this classification.
\begin{enumerate}
\item Isomorphism classes of simple functors are in bijective correspondence with partitions.
For each partition $\lambda=(\lambda_1,\dots,\lambda_k)$ we fix a simple functor $L_\lambda$ in the corresponding isomorphism class. Then $L_\lambda$ is homogeneous of degree $\sum \lambda_i$. We call $\lambda$ the \emph{highest weight} of $L_\lambda$.  For example, the only simple functor of degree $0$ is $L_{(0)}=\kk$.
\item\label{item2} Simple functors are self-dual. To be more specific, each simple functor $L$ is isomorphic to its dual $L^\sharp$, defined by $L^\sharp(V):=L(V^*)^*$.
\item\label{item3} Simple functors have endomorphism rings of dimension $1$.
\item For all partitions $\lambda$ and $\mu$, $L_{\lambda+\mu}$ is a composition factor of $L_\lambda\otimes L_\mu$. 
\end{enumerate}
\begin{remark}
Actually, one needs that $\kk$ is algebraically closed to obtain easily (Schur Lemma) that the endomorphism ring of a simple functor has dimension one. When $\kk$ is not algebraically closed, this can be proved using the fact that Schur algebras are quasi-hereditary, see e.g. \cite[Chap. 3]{Martin}. 
\end{remark}

\subsection{Frobenius twists}\label{subsec-frob-twist}
let $\kk$ be a field of positive characteristic $p$. For all $r\ge 0$, we denote by $I^{(r)}$ the $r$-th Frobenius twist functor. The functor $I^{(0)}=I=S^1=\Gamma^1=\Lambda^1$ is the identity functor. More generally, for all $r\ge 0$ the functor  $I^{(r)}$ is the unique simple additive functor of degree $p^r$ (up to isomorphism). 

\begin{notation}\label{nota-frob}
We use the traditional notation $F^{(r)}=F\circ I^{(r)}$. We also denote by $F\otimes G^{(r)}$ the tensor product of $F$ and $G^{(r)}$, i.e. Frobenius twists have a priority higher than tensor products in our notations.
\end{notation}

The effect on $\Ext^*$ of precomposition by Frobenius twist is now well-understood in all degrees \cite{TouzeClasses, Chalupnik15}. In particular, in degrees $i=0$ or $i=1$, the $\kk$-linear morphism induced by precomposition by $I^{(r)}$:
$$\Ext^i_{\PP_\kk}(F,G)\to \Ext^i_{\PP_\kk}(F^{(r)},G^{(r)}) $$
is an isomorphism. This description of the effect of precomposition by Frobenius twists in degrees $0$ and $1$ case can be proved by very elementary means, see e.g. \cite[Appendix A]{BMT}. We will not need to know about higher degrees, except in the proof of proposition \ref{prop-op}.

\subsection{Bifunctors and sum-diagonal adjunction}\label{subsec-cross}

We will need strict polynomial functors with several variables for intermediate computations, as well as in the study of tensor products of Steinberg type in section \ref{sec-tensS}. Definitions and basic properties of strict polynomial functors extend without problem to the case of functors with several variables, and we refer to \cite[section 2]{SFB}, \cite[section 2]{TouzeClassical} or \cite[section 3]{TouzeControl} for details. We recall here the main features of the theory in the context of bifunctors, and leave to the reader the obvious formulas with three variables or more.

Given two nonnegative integers $d_1$ and $d_2$, we denote by $\PP_{d_1,d_2,\kk}$ the category of homogeneous strict polynomial bifunctors of bidegree $(d_1,d_2)$ (with possibly infinite dimensional values). Typical examples of objects of this category are the \emph{bifunctors of separable type}, which are the bifunctors of the form: 
$$F\boxtimes G: (V,W)\mapsto F(V)\otimes G(W) $$
where $F$, resp. $G$, is a homogeneous strict polynomial functor of degree $d_1$, resp. $d_2$. Just as in the one variable case, evaluating bifunctors on a pair of vector spaces $(\kk^n,\kk^m)$ yields a functor:
$$\PP_{d_1,d_2,\kk}\to S(n,d_1)\otimes S(m,d_2)\text{-Mod}\;,$$
where $S(m,d_1)$ and $S(m,d_2)$ are Schur algebras (which are finite dimensional). Moreover, this functor is an equivalence of categories if $n\ge d_1$ and $m\ge d_2$. In particular $\PP_{d_1,d_2,\kk}$ satisfies the three properties mentioned at the end of section \ref{subsec-recallbasic}. We have a K\"unneth morphism:
$$\Ext^*_{\PP_{d_1,\kk}}(F_1,G_1)\otimes \Ext^*_{\PP_{d_2,\kk}}(F_2,G_2)\xrightarrow[]{\kappa} \Ext^*_{\PP_{d_1,d_2,\kk}}(F_1\boxtimes F_2,G_1\boxtimes G_2)\;.$$ 
This K\"unneth morphism is an isomorphism if the quadruple $(F_1,G_1,F_2,G_2)$ satisfies the following condition.
\begin{condition}[K\"unneth condition]\label{KCond}
In the quadruple $(F_1,G_1,F_2,G_2)$, $F_1$ and $F_2$ are both finite functors, or $F_1$ and $G_1$ are both finite functors.
\end{condition}

We also denote by $\PP_{d,\kk}(2)$ the category of homogeneous strict polynomial bifunctors of total degree $d$, and by $\PP_{\kk}(2)$ the category of strict polynomial functors of bounded degree, with possibly infinite dimensional values. We have decompositions:
$$\PP_{\kk}(2)=\bigoplus_{d\ge 0}\PP_{d,\kk}(2)\;,\qquad\PP_{d,\kk}(2)=\bigoplus_{d_1+d_2=d}\PP_{d_1,d_2,\kk}\;.$$
In particular, each bifunctor $B$ decomposes uniquely as a direct sum $B=\bigoplus B^{(d_1,d_2)}$ where $B^{(d_1,d_2)}$ is a homogeneous strict polynomial bifunctor of bidegree $(d_1,d_2)$. We shall refer to $B^{(d_1,d_2)}$ as the \emph{homogeneous component of bidegree $(d_1,d_2)$ of $B$}. 
A typical example of (degree $d$ homogeneous) bifunctor is the bifunctor
$$F_{\boxplus}:(V,W)\mapsto F(V\oplus W)$$
where $F$ is a (degree $d$ homogeneous) strict polynomial functor of degree $d$. Conversely, from a (degree $d$ homogeneous) bifunctor $B$ of total degree $d$ on can construct a (degree $d$ homogeneous) strict polynomial functor with one variable by \emph{diagonal evaluation}:
$$B_{\Delta}:V\mapsto B(V,V)\;.$$
These two constructions are exact and adjoint to each other on both sides. Hence we have graded isomorphisms:
\begin{align*}
&\Ext^*_{\PP_{\kk}(2)}(B,F_{\boxplus})\simeq \Ext^*_{\PP_{\kk}}(B_\Delta,F)\;,\\
&\Ext^*_{\PP_{\kk}(2)}(F_{\boxplus},B)\simeq \Ext^*_{\PP_{\kk}}(F,B_\Delta)\;.
\end{align*}
These two isomorphisms where first used in the context of strict polynomial functors in \cite{FFSS}. In this article, they will be the key tool for theorem \ref{thm-1}. As in \cite{FFSS}, we will often use them when $B$ is of separable type $B=G\boxtimes H$, hence when $B_\Delta=G\otimes H$.

\subsection{The internal tensor product}\label{subsec-internal}

The category $\PP_{d,\kk}$ is endowed with a closed symmetric monoidal structure. We denote this internal tensor product by $\uotimes$, and by $\uHom$ the associated internal hom. We refer the reader to \cite{Krause} for a presentation of this internal tensor product. 
We study the internal tensor product of simple functors in section \ref{sec-tensint}. For this purpose, we will use the following facts.
\begin{enumerate}
\item  If $F$ is a functor, we denote by $F^V$ the parametrized functor $W\mapsto F(\hom_{\kk}(V,W))$. Then the internal $\hom$ is the functor given by 
$$\uHom(F,G)(V)=\hom_{\PP_{d,\kk}}(F^V,G)\;. $$
\item The study of internal tensor products can be reduced to the study of internal $\hom$ by using the isomorphism natural with respect to $F,G$:
$$ (F\uotimes G)^\sharp\simeq \uHom(F,G^\sharp) \;.$$
Here $^\sharp$ is the duality defined by $F^\sharp(V)=F(V^\vee)^\vee$ where `$^\vee$' is the $\kk$-linear duality of vector spaces.
\end{enumerate}
\begin{remark}
Schur algebras do not have a Hopf algebra structure in general. (Indeed, Schur algebras have finite global dimension, and a Hopf algebra structure would make them  self-injective in addition, hence semi-simple.) Thus the internal tensor product on $\PP_{d,\kk}$ is an example of monoidal product which does not come from a Hopf algebra structure.
\end{remark}

\subsection{Connection with representations of symmetric groups}

Strict polynomial functors are related to representations of the symmetric groups $\Si_d$ via the Schur functors. We will use these Schur functors in sections \ref{sec-tensint} and \ref{sec-applicsym}.
Let $d$ be a positive integer. Consider the right action of the symmetric group $\Si_d$ on $\otimes^d$ given by permuting the factors of the tensor product. The Schur functor is the functor:
$$f_d:=\hom_{\PP_{d,\kk}}(\otimes^d,-): \PP_{d,\kk}\to \kk\Si_d\text{-Mod}\;.$$
Since $\otimes^d$ is projective, the Schur functor $f_d$ is exact. It has adjoints on both sides. To be more specific, the left adjoint $\ell_d$ is defined by $\ell_d(M)=(\otimes^d)\otimes_{\Si_d}M$, while the right adjoint $r_d$ is defined by $r_d(M)=((\otimes^d)\otimes M)^{\Si_d}$. The unit and counit of adjunction induce natural isomorphisms
$$M\xrightarrow[]{\simeq}f_d(\ell_d(M))\;, \quad f_d(r_d(M))\xrightarrow[]{\simeq} M\;. $$
In particular, the Schur functor $f_d$ is a quotient functor.

\section{Exts in low degrees between tensor products}\label{sec-tensExt}

\subsection{Definition of $i(F,r)$ and $i(p,r)$}
For all tuples $\lambda=(\lambda_1,\dots,\lambda_n)$ of nonnegative integers, we let:
$$\Gamma^\lambda:=\Gamma^{\lambda_1}\otimes\dots\otimes \Gamma^{\lambda_n}\,,\text{ and }\;S^\lambda:= S^{\lambda_1}\otimes\dots\otimes S^{\lambda_n}\;.$$
Let $\mathcal{T}$ denote the set of all tuples of nonnegative integers. Then the family $(\Gamma^\lambda)_{\lambda\in\mathcal{T}}$ forms a projective generator of $\PP_\kk$, while the family $(S^\lambda)_{\lambda\in\mathcal{T}}$ forms an injective cogenerator of $\PP_\kk$.

\begin{definition}\label{def-bounded}
Let $r$ be a nonnegative integer.
A tuple of nonnegative integers $\lambda=(\lambda_1,\dots,\lambda_n)$ is \emph{$p^r$-bounded} if $\lambda_k<p^r$ for all $k$. A \emph{basic $p^r$-bounded projective} (resp. injective) is a functor of the form $\Gamma^\lambda$ (resp. $S^\lambda$) where $\lambda$ is  $p^r$-bounded.
A strict polynomial functor $F$ is \emph{left $p^r$-bounded} if it is a quotient of a direct sum of basic $p^r$-bounded projectives.
Similarly, $F$ is \emph{right $p^r$-bounded} if it embeds in a product of basic $p^r$-bounded injectives.
\end{definition}

\begin{remark}
If $r=0$, the tuples $(0,\dots,0)$ are the only $p^r$-bounded tuples. Since $\Gamma^{0}=S^0=\kk$, a functor is $p^0$-bounded if and only if it is constant.
\end{remark}

The following lemma collects elementary facts on $p^r$-bounded functors.
\begin{lemma}\label{lm-elem}
\begin{enumerate}
\item[1.] The following statements are equivalent:
\begin{enumerate}
\item[(i)] $F$ is right $p^r$-bounded,
\item[(ii)] $\Soc(F)$ is right $p^r$-bounded,
\item[(iii)] $F$ embeds into a direct sum of basic $p^r$-bounded injectives.
\end{enumerate}
\item[2.] The following statements are equivalent:
\begin{enumerate}
\item[(i')] $F$ is left $p^r$-bounded,
\item[(ii')] $\mathrm{Head}(F)$ is left $p^r$-bounded,
\item[(iii')] $F$ is the union of finite left $p^r$-bounded subfunctors.
\end{enumerate}

\end{enumerate}
\end{lemma}
\begin{proof}
1. It is clear that (iii)$\Rightarrow$(i)$\Rightarrow$(ii). If (ii) holds, then each simple summand of $Soc(F)$ embeds into a basic $p^r$-bounded injective, thus $\Soc(F)$ embeds into a direct sum of basic $p^r$-bounded injectives $J$. Since $J$ is injective, the monomorphism $\Soc(F)\hookrightarrow J$ extends to a map $\phi:F\to J$. But $\Soc(\ker\phi)\subset \ker\phi\cap \Soc(F)=0$ so $\phi$ is injective. This proves (iii).

2. It is clear that (i')$\Rightarrow$(ii'), the proof of  (ii')$\Rightarrow$(i') is dual to the one of (ii)$\Rightarrow$(iii). Let us prove (i')$\Leftrightarrow$(iii').
If $F$ is left $p^r$-bounded, there is a map $\pi:\bigoplus_{\lambda\in A} \Gamma^\lambda\twoheadrightarrow F$. Thus $F$ is the union of the images of the $\pi(\bigoplus_{\lambda\in B}\Gamma^\lambda)$ where $B$ is a finite subset of $A$. Conversely, if $F$ is the union of a family of finite left $p^r$-bounded functors $F_\alpha$, then $F$ is a quotient of $\bigoplus F_\alpha$, hence $F$ is left $p^r$-bounded. 
\end{proof}

\begin{definition}\label{def-pi}
Let $r$ be a nonnegative integer, and let $F$ be a functor.
\begin{enumerate}
\item We define $p(F,r)\in [0,+\infty]$ as the supremum of all the integers $n\ge 0$ such that $F$ admits a projective resolution $P$ in which the first $n$ objects $P_0,\dots,P_{n-1}$ are left $p^r$-bounded.
\item We define $i(F,r)\in [0,+\infty]$ as the supremum of all the integers $n\ge 0$ such that $F$ admits an injective resolution $J$ in which the first $n$ objects $J^0,\dots,J^{n-1}$ are right $p^r$-bounded.
\end{enumerate}
\end{definition}

\begin{remark}
\begin{enumerate}
\item[i)] By definition $p(F,r)> 0$ (resp. $i(F,r)> 0$) if and only if $F$ is left (resp. right) $p^r$-bounded.
\item[ii)] If $p^r>\deg F$, then all projectives or injectives appearing in any resolution of $F$ are $p^r$-bounded, hence $p(F,r)=i(F,r)=+\infty$. In particular if $F$ is constant, it is homogeneous of degree $0$ and $i(F,r)=p(F,r)=+\infty$ for all $r\ge 0$.
\item[iii)] In the definition $p(F,r)$ and $i(F,r)$ belong to $[0,+\infty]$. However, the category $\PP_{d,\kk}$ has finite global dimension $\mathrm{gldim}(d,\kk)$, which is explicitly computed in \cite{Totaro}. If $F$ is homogeneous of degree $d$, then $p(F,r)$ and $i(F,r)$ actually belong to $[0,\dots,\mathrm{gldim}(d,\kk)-1]\cup \{+\infty\}$.
\end{enumerate}
\end{remark}

\subsection{Application to the connectedness of cup products}
The tensor product $\otimes:\PP_\kk\times\PP_\kk\to \PP_\kk$, induces a cup product on extension groups in the usual way, see e.g. \cite[3.2]{Benson}. The purpose of this section is to prove the following result. 

\begin{theorem}\label{thm-1}
Let $(F,G,X,Y)$ be a quadruple of homogeneous strict polynomial functors satisfying the K\"unneth condition \ref{KCond}, and let $r\ge 0$. The cup product induces a graded injective map:
$$\Ext^*_{\PP_\kk}(F,G)\otimes \Ext^*_{\PP_\kk}(X^{(r)},Y^{(r)})\hookrightarrow \Ext^*_{\PP_\kk}(F\otimes X^{(r)},G\otimes Y^{(r)})\;.$$
Moreover, this graded injective map is an isomorphism in degree $k$ in the following situations.
\begin{enumerate}
\item  When $\deg F < \deg G$, and $k<i(G,r)$.
\item  When $\deg F > \deg G$, and $k<p(F,r)$.
\item  When $\deg F = \deg G$, and $k<p(F,r)+i(G,r)$.
\end{enumerate}
\end{theorem}

The remainder of section \ref{sec-tensExt} is devoted to the proof of theorem \ref{thm-1}. Observe that we have a factorization of cup products 
$$\xymatrix{
\Ext^*_{\PP_\kk}(F,G)\otimes \Ext^*_{\PP_\kk}(X^{(r)},Y^{(r)})\ar[d]^-{\kappa}\ar[r]^-{\cup}&\Ext^*_{\PP_\kk}(F\otimes X^{(r)},G\otimes Y^{(r)})\\
\Ext^*_{\PP_\kk(2)}(F\boxtimes X^{(r)},G\boxtimes Y^{(r)})\ar[ru]_-{-_\Delta}
}\;.$$
In particular, theorem \ref{thm-1} is a consequence of the following slightly more general statement, in which the K\"unneth condition is removed.
\begin{theorem}\label{thm-1-ter}
Let $F,G,X,Y$ be homogeneous functors, and let $r\ge 0$. Diagonal evaluation induces a graded injective map:
$$\Ext^*_{\PP_\kk(2)}(F\boxtimes X^{(r)},G\boxtimes Y^{(r)})\hookrightarrow \Ext^*_{\PP_\kk}(F\otimes X^{(r)},G\otimes Y^{(r)})\;.$$
Moreover, this graded injective map is an isomorphism in degree $k$ in the situations listed in theorem \ref{thm-1}.
\end{theorem}

The proof of theorem \ref{thm-1-ter} relies on a series of lemmas. The proofs of these lemmas are all based upon the sum-diagonal adjunction technique recalled in section \ref{subsec-cross}.
\begin{lemma}\label{lm-cup}
Let $F,G,F',G'$ be homogeneous functors satifying $\deg F=\deg G$ and $\deg F'=\deg G'$. Diagonal evaluation yields an injective map
$$\Ext^*_{\PP_\kk(2)}(F\boxtimes F',G\boxtimes G')\hookrightarrow \Ext^*_{\PP_\kk}(F\otimes F',G\otimes G')$$
whose cokernel is isomorphic to following direct sum, indexed by the tuples of nonnegative integers $(d_1,d_2,e_1,e_2)$ such that $d_2>0$ and $e_1>0$:
$$\bigoplus_
{\text{\footnotesize
$
\begin{array}{c}
0< d_2,e_1\\
0\le d_1,e_2
\end{array}
$
}}
\Ext^*_{\PP_\kk(2)}\left(F\boxtimes F', (G_\boxplus)^{(d_1,d_2)}\otimes (G'_\boxplus)^{(e_1,e_2)}\right)\;. $$
This cokernel is also isomorphic to the direct sum:
$$\bigoplus_
{\text{\footnotesize
$
\begin{array}{c}
0< d_2,e_1\\
0\le d_1,e_2
\end{array}
$
}}
\Ext^*_{\PP_\kk(2)}\left((F_\boxplus)^{(d_1,d_2)}\otimes (F'_\boxplus)^{(e_1,e_2)},G\boxtimes G'\right)\;. $$
\end{lemma}
\begin{proof}
We recall the proof of injectivity from \cite{TouzeClassical} and we prove the first description of the cokernel. The proof of the second description is similar. The map given by diagonal evaluation is equal to the composite of the map
$$\eta_*:\Ext^*_{\PP_\kk(2)}(F\boxtimes F',G\boxtimes G')\to \Ext^*_{\PP_\kk(2)}(F\boxtimes F',(G\otimes G')_\boxplus)$$
induced by the canonical map $\eta:G\boxtimes G'\to (G\otimes G')_\boxplus$, together with the adjunction isomorphism
$$\Ext^*_{\PP_\kk(2)}(F\boxtimes F',(G\otimes G')_\boxplus)\simeq \Ext^*_{\PP_\kk}(F\otimes F',G\otimes G')\;.$$ 
Thus, to prove lemma \ref{lm-cup}, it suffices to prove that $\eta_*$ is injective and to identify its cokernel. But $(G\otimes G')_\boxplus = G_\boxplus\otimes G'_\boxplus$, and there is a decomposition
$$(G\otimes G')_\boxplus = G\boxtimes G' \oplus \bigoplus_{d_2>0 \text{ or } e_1>0} (G_\boxplus)^{(d_1,d_2)}\otimes (G'_\boxplus)^{(e_1,e_2)}\;.$$
The map $\eta$ identifies with the inclusion of $G\boxtimes G'$ into the right hand side, and since the decomposition is a direct sum, $\eta$, hence $\eta_*$ admits a section, and the cokernel of $\eta_*$ is isomorphic to:
$$\bigoplus_{d_2>0 \text{ or } e_1>0} \Ext^*_{\PP_\kk(2)}(F\boxtimes F',  (G_\boxplus)^{(d_1,d_2)}\otimes (G'_\boxplus)^{(e_1,e_2)})\;.$$
This is almost the description of the cokernel given in lemma \ref{lm-cup}, but the summation index is different. Since there are no nonzero extensions between homogeneous bifunctors of different degrees, all the terms in the direct sum are zero, except the ones satisfying 
$d_1+e_1=\deg F $ and $d_2+e_2=\deg F'$. Since $d_1+d_2=\deg G=\deg F$, the nonzero terms in the direct sum satisfy $e_1=d_2$. Thus we can replace the summation index `$d_2>0$ or $e_1>0$' by `$e_1>0$ and $d_2>0$' and we are done. 
\end{proof}

The proof of the next lemma is omitted since it is very similar to the proof of lemma \ref{lm-cup}.
\begin{lemma}\label{lm-cok}
Let $F,F',G,G'$ be homogeneous functors. If $\deg F>\deg G$,
then  $\Ext^*_{\PP_\kk}(F\otimes F',G\otimes G')$ is isomorphic to the following direct sum, indexed by the tuples of nonnegative integers $(d_1,d_2,e_1,e_2)$ such that $e_1>0$:
$$\bigoplus_{\text{\footnotesize
$
\begin{array}{c}
0< e_1\\
0\le d_1,d_2,e_2
\end{array}
$
}}\Ext^*_{\PP_\kk}(F\boxtimes F',(G_\boxplus)^{(d_1,d_2)}\otimes (G'_\boxplus)^{(e_1,e_2)})\;.$$
If $\deg F<\deg G$ then it is isomorphic to:
$$\bigoplus_{\text{\footnotesize
$
\begin{array}{c}
0< e_1\\
0\le d_1,d_2,e_2
\end{array}
$
}}
\Ext^*_{\PP_\kk}((F_\boxplus)^{(d_1,d_2)}\otimes (F'_\boxplus)^{(e_1,e_2)},G\boxtimes G')\;.$$
\end{lemma}

The next two vanishing lemmas are analogues of the key vanishing result (i.e. Pirashvili's vanishing) of \cite[Thm 2.13]{FS}. 
\begin{lemma}\label{lm-annul-prelim}
Let $F$ and $G$ be homogeneous functors with $\deg G >0$ and let $\lambda$ be a $p^r$-bounded tuple. Then 
$$\hom_{\PP_\kk}(F\otimes G^{(r)}, S^\lambda)=0=\hom_{\PP_\kk}(\Gamma^\lambda,F\otimes G^{(r)})\;.$$
\end{lemma}
\begin{proof}
We prove the first equality, the proof of the second one is similar. We will use that for all homogeneous $G$ of positive degree and for all $p^r$-bounded tuple $\nu$, 
$$\hom_{\PP_\kk}(G^{(r)},S^\nu)=0\;. \qquad(*)$$
This is proved when $G$ has finite dimensional values in \cite[lm 2.3]{TouzeTroesch}, and it holds for an arbitrary $G$ because any functor is the colimit of its finite subfunctors. (Alternatively, one could also prove this vanishing by sum-diagonal adjunction.)
To reduce the equality of lemma \ref{lm-annul-prelim} to formula $(*)$, we proceed as follows. First, sum-diagonal adjunction yields an isomorphism:
$$\hom_{\PP_\kk}(F\otimes G^{(r)}, S^\lambda)\simeq \hom_{\PP_\kk}(F\boxtimes G^{(r)}, (S^\lambda)_\boxplus)\;.$$
We observe that $(S^\lambda)_\boxplus$ decomposes as a direct sum of tensor products of the form $S^\mu\boxtimes S^\nu$ with $\mu$ and $\nu$ $p^r$ bounded. Thus lemma \ref{lm-annul-prelim} will be proved if we can prove that $\hom_{\PP_\kk}(F\boxtimes G^{(r)},S^\mu\boxtimes S^\nu)$ is zero when $\mu$ and $\nu$ are $p^r$-bounded. So let $\phi:F\boxtimes G^{(r)}\to S^\mu\boxtimes S^\nu$ be a morphism. By freezing the first variable of the bifunctors, we obtain for all $V$ a morphism of functors $\phi_V:F(V)\otimes G^{(r)}(-)\to S^\mu(V)\otimes S^\nu(-)$. By formula $(*)$, $\phi_V$ must be zero for all $V$. In particular, $\phi$ must be zero.
\end{proof}

\begin{lemma}\label{lm-annul}
Let $r$ be a positive integer, let $J$ be a be a right $p^r$-bounded injective functor, let $P$ be a left $p^r$-bounded projective functor, let $Z$ be a homogeneous functor let $B$ and $C$ be two homogeneous bifunctors. If $\deg C=(e_1,e_2)$ with $e_1>0$, and $C^{(r)}$ denotes the bifunctor $(V,W)\mapsto C(V^{(r)},W^{(r)})$ then  
$$\Ext^*_{\PP_\kk(2)}(B\otimes C^{(r)},J\boxtimes Z)=0= \Ext^*_{\PP_\kk(2)}(P\boxtimes Z,B\otimes C^{(r)})\;.$$
\end{lemma}
\begin{proof}
We will prove the first equality, the proof of the second one is similar. 
If $J_Z$ is an injective resolution of the functor $Z$, then $J\boxtimes J_Z$ is an injective resolution of the bifunctor $J\boxtimes Z$. Thus, it is sufficient to do the proof in degree zero (i.e. for $\hom$) and when $Z$ is injective, the general case will follow by taking resolutions. 
So let us take a morphism of bifunctors $\phi:B\otimes C^{(r)}\to J\boxtimes Z$. Then by freezing the first variable of the bifunctors, we obtain for all $V$ a morphism of functors:
$$\phi_V:B(V,-)\otimes C^{(r)}(V,-)\to J(V)\otimes Z(-)\;.$$
But by lemma \ref{lm-annul-prelim} $\phi_V$ is zero for all $V$. In particular, $\phi$ must be zero.
\end{proof}

\begin{proof}[Proof of theorem \ref{thm-1-ter}]
By lemma \ref{lm-cup}, diagonal evaluation yields an injective morphism on the $\Ext$-level (if $\deg F\ne \deg G$ or $\deg X\ne \deg Y$ the source of the cup product morphism is zero for degree reasons, so that injectivity is trivial). Hence, it remains to prove the cancellation in low degrees of the cokernel, described in lemmas \ref{lm-cup} and \ref{lm-cok}.

Assume that $\deg F\ge \deg G$. Take a finite resolution of $F$ of the form:
$$ 0\to \overline{F}\to F_{p(F,r)-1}\to \dots\to F_0\to F\to 0 $$
where the functors $F_k$ with $k< p(F,r)$ are right $p^r$-bounded projective functors. Take $B$ and $C$ as in lemma \ref{lm-annul}. By using long exact sequences, we obtain that for all $k\in\Z$ (with the convention that $\Ext$ are zero in nonpositive degrees):
$$\Ext^{*}_{\PP_\kk(2)}(F\boxtimes X^{(r)},B\otimes C^{(r)})\simeq \Ext^{*-p(F,r)}_{\PP_\kk(2)}(\overline{F}\boxtimes X^{(r)},B\otimes C^{(r)})\;.\quad(*)$$
In particular the $\Ext$ on the left hand side are $(p(F,r)-1)$-connected, i.e. zero in degrees $*<p(F,r)$. By lemmas \ref{lm-cup} and \ref{lm-cok}, the case where $B=(F_\boxplus)^{(d_1,d_2)}$ and $C^{(r)}=(Y^{(r)}_\boxplus)^{(e_1,e_2)}$ with $e_1>0$ implies that the cup product induces an isomorphism in degrees less than $p(F,r)$. A similar argument shows that the cup product is an isomorphism in degrees less than $i(G,r)$ if $\deg F\le \deg G$. 
Assume now that $\deg F=\deg G$. By lemma \ref{lm-cup} and isomorphism $(*)$, the statement of theorem \ref{thm-1-ter} is equivalent to showing that 
$$\Ext^*_{\PP_\kk(2)}\left(\overline{F}\boxtimes X^{(r)}, (G_\boxplus)^{(d_1,d_2)}\otimes (Y^{(r)}_\boxplus)^{(e_1,e_2)}\right)$$
is $(i(G,r)-1)$-connected for $d_2>0$ and $e_1>0$. By lemma \ref{lm-cup} again, this is equivalent to showing that the cup product
$$\Ext^*_{\PP_\kk}(\overline{F},G)\otimes\Ext^*_{\PP_\kk}(X^{(r)},Y^{(r)})\to \Ext^*_{\PP_\kk}(\overline{F}\otimes X^{(r)},G\otimes Y^{(r)})$$
is an isomorphism in degrees less than $i(G,r)$. But we have already proved that the latter holds since $\deg \overline{F}\le \deg G$.
\end{proof}

\section{An equivalent definition of $p(F,r)$ and $i(F,r)$}\label{sec-eq-def}
Next proposition gives an equivalent definition of $p(F,r)$ and $i(F,r)$. While the proof of theorem \ref{thm-1} really relies on definition \ref{def-bounded} given before, this new definition is useful to apply theorem \ref{thm-1} in concrete situations. In particular, the translation of theorem \ref{thm-1} in low degrees given in corollary \ref{thm-1bis} below will be used in sections \ref{sec-tensS} and \ref{sec-tensint}. 

Recall that a partition $\lambda=(\lambda_1,\dots,\lambda_n)$ is \emph{$p^r$-restricted} (for some nonnegative integer $r$) if $\lambda_n< p^r$ and for $i<n$, $\lambda_i-\lambda_{i+1}<p^r$. By convention the partition $(0)$ is $p^r$-restricted for all $r\ge 0$. Using euclidean division, one sees that any partition $\lambda$ can be written in a unique way as a sum $\lambda=\lambda^{0}+p^r\lambda^1$, where $\lambda^0$ is $p^r$-restricted. A simple indexed by a $p^r$-restricted partition will be loosely called a $p^r$-restricted simple. The proof of proposition \ref{prop-eq} relies on two classical fundamental results on simple polynomial representations in positive characteristic. We state them here with references to the literature, but we prove in appendix \ref{app-Stein} that both of them can actually be derived from theorem \ref{thm-1}. 
\begin{enumerate}
\item Steinberg tensor product theorem \cite[II.3.17]{Jantzen}. If $\lambda$ is a $p^r$-restricted  and $\mu$ is an arbitrary partition, then $L_\lambda\otimes L_\mu^{(r)}$ is isomorphic to $L_{\lambda+p^r\mu}$.
\item Clausen and James' theorem \cite{Clausen,James}. A partition $\lambda$ of $d$ is $p$-restricted if and only if $\hom_{\PP}(L_\lambda,\otimes^d)=\hom_{\PP}(\otimes^d,L_\lambda)$ is nonzero.
\end{enumerate}

\begin{proposition}\label{prop-eq}
Let $r$ be a nonnegative integer, and let $F$ be a functor.
\begin{enumerate}
\item The integer $p(F,r)$ is the supremum of all $n\ge 0$ such that $F$ admits a projective resolution $P$, in which the first $n$ objects $P_0,\dots,P_{n-1}$ are direct sums of projective covers of $p^r$-restricted simples.
\item The integer $i(F,r)$ is the supremum of all $n\ge 0$ such that $F$ admits an injective resolution $J$, in which the first $n$ objects $J^0,\dots,J^{n-1}$ direct sums of injective envelopes of $p^r$-restricted simples. 
\end{enumerate}
\end{proposition}
\begin{proof}
We restrict ourselves to proving the second statement, the proof of the first one is similar. Let us denote by $J_\mu$ the injective envelope of $L_\mu$. We have to prove that (i) for all $p^r$ restricted partition $\mu$, there is a $p^r$-bounded tuple $\lambda$ such that $J_\mu$ is a direct summand of $S^\lambda$, and (ii) for all $p^r$-bounded tuple $\lambda$, the indecomposable direct summands of $S^\lambda$ are all isomorphic to $J_\mu$ with $\mu$ $p^r$-restricted. 

Write $\mu=\sum_{i=0}^n p^i\mu^i$ for $p$-restricted partitions $\mu^i$. By Steinberg tensor product theorem $L_\mu$ is isomorphic to $L_{\mu^0}^{(0)}\otimes\dots\otimes L_{\mu^n}^{(n)}$. By Clausen and James' theorem $L_\mu$ is a subfunctor of $(I^{(0)})^{\otimes |\mu^0|}\otimes \dots\otimes (I^{(n)})^{\otimes |\mu^n|}$. Since $I^{(i)}\subset S^{p^i}$, we obtain that $L_\mu$ is a subfunctor of $\bigotimes_{0\le i\le n}(S^{p^i})^{\otimes |\mu^i|}$.
If $\mu$ is $p^r$-restricted then $n<r$, so $L_\mu$ (hence also $J_\mu$) is a subfunctor of $S^\lambda$ with $\lambda$ $p^r$-bounded. This proves (i). 
Let $\lambda$ be a $p^r$-bounded tuple, and let $\mu$ be a partition such that $\mu$ is not $p^r$-restricted. By Steinberg tensor product theorem, $L_\mu$ is isomorphic to $L_{\mu'}\otimes L_{\mu''}^{(r)}$ for a $p^r$-restricted partition $\mu'$ and a nonzero partition $\mu''$. S by lemma \ref{lm-annul-prelim} $\hom_\PP(L_\mu,S^\lambda)$ is zero, hence $J_\mu$ is not a direct summand of $S^\lambda$. This proves (ii).
\end{proof}

\begin{corollary}\label{cor-bounded}
For all $F$, $i(F,r)>0$ (resp. $p(F,r)>0$) if and only if $\mathrm{Head}(F)$ (resp. $\Soc(F)$) is a direct sum of $p^r$-restricted simples.
\end{corollary}

If $(d_0,\dots,d_k)$ is a tuple of nonnegative integers, we let 
$$T^{(d_0,\dots,d_k)}=(\otimes^{d_0})\otimes(\otimes^{d_1})^{(1)}\otimes\dots \otimes (\otimes^{d_k})^{(k)}$$
 
\begin{corollary}\label{cor-caract}
Let $L$ be a simple functor. The following are equivalent: (i) $p(L,r)>0$, (ii) $i(L,r)>0$, (iii) $L$ is $p^r$ restricted, (iv) there exists a tuple $(d_0,\dots d_{r-1})$ such that $L$ is a quotient of $T^{(d_0,\dots,d_{r-1})}$, and (v) there exists a tuple $(d_0,\dots d_{r-1})$ such that $L$ embeds into $T^{(d_0,\dots,d_{r-1})}$.
\end{corollary}
\begin{proof}
(i)$\Leftrightarrow$ (ii) $\Leftrightarrow$ (iii) by corollary \ref{cor-bounded}, and (iv) $\Leftrightarrow$ (v) by Kuhn duality (both $L$ and $T^{(d_0,\dots,d_{r-1})}$ are self-dual). (iii) $\Rightarrow$ (iv) by Steinberg tensor product theorem and Clausen and James' theorem (as already used in the proof of proposition \ref{prop-eq}. It remains to prove that (iv)$\Rightarrow$(iii). Assume that $L$ is not $p^r$-restricted, then by Steinberg tensor product theorem $L\simeq L_1\otimes L_2^{(r)}$ where $L_1$ is $p^r$-restricted and $L_2$ is a nontrivial simple functor. Thus by proposition \ref{prop-eq} we may apply corollary \ref{thm-1bis} with `$F$'=$L_1$, `$G$'=$T^{(d_0,\dots,d_{r-1})}$, `$X$'=$L_2$ and `$Y$'=$\kk$ to obtain that there is no nonzero map $L\to T^{(d_0,\dots,d_{r-1})}$. Thus all simple subfunctors of $T^{(d_0,\dots,d_{r-1})}$ are $p^r$-restricted.
\end{proof}

Next corollary is a translation of theorem \ref{thm-1} in low degrees in terms of $p^r$-restricted weights. It will be used in sections \ref{sec-tensS} and \ref{sec-tensint}.

\begin{corollary}\label{thm-1bis}
Let $(F,G,X,Y)$ be a quadruple of homogeneous strict polynomial functors satisfying the K\"unneth condition \ref{KCond}, and let $r\ge 0$. Precomposing by $I^{(r)}$ and taking cup products induces injective morphisms 
\begin{align}
&\hom_{\PP_{\kk}}(F,G)\otimes \hom_{\PP_{\kk}}(X,Y)\hookrightarrow\hom_{\PP_{\kk}}(F\otimes X^{(r)},G\otimes Y^{(r)})\label{eq-inj1ter}\;,\\
&\begin{array}{c}\hom_{\PP_{\kk}}(F,G)\otimes \Ext^1_{\PP_{\kk}}(X,Y)\\
\oplus\,\Ext^1_{\PP_{\kk}}(F,G)\otimes \hom_{\PP_{\kk}}(X,Y)  
\end{array}
\hookrightarrow\Ext^1_{\PP_{\kk}}(F\otimes X^{(r)},G\otimes Y^{(r)})\;. \label{eq-inj2ter}
\end{align}
If one of the conditions C1 or C2 below holds, morphism  \eqref{eq-inj1ter} is an isomorphism. If both C1 and C2 hold, then morphism \eqref{eq-inj2ter} is also an isomorphism. 
\begin{enumerate}
\item[] (C1) \;$\deg F\le \deg G$ and $\mathrm{Head}(G)$ is a direct sum of $p^r$-restricted simples.
\item[] (C2) \;$\deg F\ge \deg G$ and $\mathrm{Soc}(F)$ is a direct sum of $p^r$-restricted simples.
\end{enumerate}
\end{corollary}
\begin{proof}
Recall from section \ref{subsec-frob-twist} that precomposing by $I^{(r)}$ yields a $\kk$-linear isomorphism on the level of $\hom$ and $\Ext^1$. Thus the statement of corollary \ref{thm-1bis} is equivalent to the statement where $X$ and $Y$ are replaced by $X^{(r)}$ and $Y^{(r)}$ at the source of the cup product maps. By corollary \ref{cor-bounded}, (C1), resp. (C2) means that $i(G,r)>0$, resp. $i(F,r)>0$. Thus corollary \ref{cor-bounded} follows directly from theorem \ref{thm-1}.
\end{proof}

\begin{remark}\label{rk-Jantzen}
In sections  \ref{sec-tensS} and \ref{sec-tensint} we will use corollary \ref{thm-1bis} for quite general functors $F$ and $G$. However, this result is already interesting in the very special case where $F$ and $G$ are $p^r$-restricted simples. Indeed the isomorphism given by corollary \ref{thm-1bis} is then a stronger form, albeit valid only for stable polynomial representations of $GL_n$, of formulas of Donkin \cite{DonkinExt} and Andersen \cite{AndersenExt}, see also \cite[II.10.16, II.10.17]{Jantzen}. For example, corollary \ref{thm-1bis} implies that if $\lambda\ne \lambda'$ are partitions of $d$ and $G=GL_n$ with $n$ big enough, then the number of $L_\mu$ in the socle of the tensor product $\Ext^1_{G_r}(L_\lambda,L_{\lambda'})^{(-r)}\otimes L_\mu'$ is zero, unless $\mu=\mu'$ in which case it equals the dimension of
 $\Ext^1_G(L_\lambda,L_{\lambda'})$.
\end{remark}

\section{Tensor products of Steinberg type}\label{sec-tensS}

Recall that a simple functor $L$ is a composition factor of an arbitrary functor $F$ if $L$ is a subquotient of $F$. Equivalently, if $0=F^{-1}\subset F^0\subset \dots \subset F$ is an exhaustive filtration of $F$ whose successive quotients are semisimple, then $L$ appears as a direct summand in one of these successive quotients.

\begin{definition}
A tensor product of Steinberg type is a strict polynomial functor isomorphic to a tensor product $F\otimes G^{(r)}$, where $r$ is a nonnegative integer and $F$ is a functor whose composition factors are all $p^r$-restricted.
\end{definition}
The purpose of the present section is to explore the structure of these tensor products of Steinberg type. Note that by Steinberg tensor product theorem (applied to the tensor product of the socle filtration of $F$ by the socle filtration of $G^{(r)}$), all composition factors of $F\otimes G^{(r)}$ are of the form $L_\lambda\otimes L_\mu^{(r)}$, with $
L_\lambda$ a composition factor of $F$ and $L_\mu$ a composition factor of $G$. This observation motivates the following definition.
\begin{definition}
Let $e$, $d$, $r$ be nonnegative integers. We let $\St(d,e,r)$ be the full subcategory of $\PP_{d+ep^r,\kk}$ supported by the functors whose composition factors are all of the form $L_\lambda\otimes L_\mu^{(r)}$ for $p^r$-restricted partition $\lambda$ of $d$ and partitions $\mu$ of $e$.
\end{definition}

\begin{lemma}\label{lm-loc}
The category $\St(d,e,r)$ contains all the tensor products of Steinberg type $F\otimes G^{(r)}$ where $F$ is homogeneous of degree $d$ and $G$ is homogeneous of degree $e$. Moreover, it is localizing and colocalising, i.e. it is closed under sums, products, subobjects, quotients and extensions.
\end{lemma}
\begin{proof}
Everything is straightforward from the definition of $\St(d,e,r)$ except maybe that $\St(d,e,r)$ is closed under arbitrary products. Let  $L$ be a composition factor of a product $\prod X_\alpha$. Then there is a nonzero map $P_L\to \prod X_\alpha$, where $P_L$ denotes the projective cover of $L$. Thus there is an $\alpha$ such that $\hom_{\PP_\kk}(P_L,X_\alpha)\ne 0$, so that $L=L_\lambda\otimes L_\mu^{(r)}$ with $\lambda$ $p^r$-restricted.  
\end{proof}

Next lemma makes a critical use of corollary \ref{thm-1bis}.

\begin{lemma}\label{lm-pres-St}
In the category $\St(d,e,r)$, any object $X$ has a presentation $P_1\to P_0\to X\to 0$ in which the $P_i$ are direct sums of tensor products of Steinberg type $F\otimes G^{(r)}$ with $F$ and $G$ finite. Similarly, $X$ has a copresentation $0\to X\to Q^0\to Q^1$ in which the $Q^i$ are products of such tensor products.
\end{lemma}
\begin{proof}
It suffices to prove that all the objects of $\St(d,e,r)$ are quotients of direct sums of tensor products of Steinberg type with values in finite dimensional spaces (then using the duality $^\sharp$, they will also embed into products of such functors). Let $X$ be an object of $\St(d,e,r)$, and let $X^i$ denote the $i$-th term of the socle filtration of $X$. Assume that $X^{i-1}$ is quotient of $P^{i-1}$, where $P^{i-1}$ has the required form. Then $X^i/X^{i-1}$ is a direct sum of $L_\lambda\otimes L_\mu^{(r)}$, and each of these functors is a homomorphic quotient of $P_\lambda\otimes P_\mu^{(r)}$ where $P_\mu$ and $P_\lambda$ are projective functors, and $P_\lambda$ is left $p^r$-bounded. Using corollary \ref{thm-1bis} and the projectivity of $P_\mu$ and $P_\lambda$, we obtain $\Ext^1_{\PP_\kk}(P_\lambda\otimes P_\mu^{(r)}, X^{i-1})=0$. Hence the map $P_\lambda\otimes P_\mu^{(r)}\to X^i/X^{i-1}$ lifts to a map $f: P_\lambda\otimes P_\mu^{(r)}\to X^i$. The functor $P_\lambda$ has a unique largest quotient 
$P'
_\lambda$ whose composition factor are $p^r$-restricted. Let $K_\lambda$ be the kernel of the quotient map $P_\lambda\twoheadrightarrow P'_\lambda$. By corollary \ref{thm-1bis}, $\hom_{\PP_\kk}(K_\lambda\otimes P_{\mu}^{(r)},X_{i})=0$. Thus $f$ induces a map $P_\lambda'\otimes P_\mu^{(r)}\to X^i$. In particular, if we define $P^{i}:=P^{i-1}\oplus \bigoplus P_\lambda'\otimes P_\mu^{(r)}$, then $P^i$ is a direct sum of tensor products of Steinberg type with values in finite dimensional vector spaces, and $X^i$ is a quotient of $P^i$. Since homogeneous strict polynomial functors have finite socle filtrations, this proves the lemma. 
\end{proof}

We will prove that the categories $\St(d,e,r)$ have an alternative description in terms of bifunctors. To be more specific, we denote by 
$$\Phi:\PP_{d,e,\kk}(2)\to \PP_{d+p^re,\kk}$$
the functor such that $\Phi(B)(V)= B(V,V^{(r)})$. We observe that $\Phi$ is exact, but it is not an equivalence of categories. For example if $d=p^r$ and $e=1$ the bifunctor $I^{(r)}\boxtimes I$ is simple, while its image $\otimes^{2(r)}$ is not ($\Lambda^{2(r)}$ is a proper subfunctor).  However $\Phi$ behaves better if we suitably restrict its source and target categories.
\begin{definition}
Let $e$, $d$, $r$ be nonnegative integers. We denote by $\StBif$ the full subcategory of $\PP_{d,e,\kk}(2)$ supported by the bifunctors whose composition factors are all of the form $L_\lambda\boxtimes L_\mu$ where $\lambda$ is a $p^r$-restricted partition of weight $d$ and $\mu$ is a partition of weight $e$
\end{definition}

We have the following analogues of lemma \ref{lm-loc} and lemma \ref{lm-pres-St}.
\begin{lemma}\label{lm-loc2}
The subcategory $\StBif$ contains all the separable functors $F\boxtimes G$ where $F$ is homogeneous of degree $d$ with $p^r$ restricted composition factors, and $G$ is homogeneous of degree $e$. Moreover, $\StBif$ is closed under sums, products, subobjects, quotients and extensions.
\end{lemma}

\begin{lemma}\label{lm-pres-St2}
In the category $\St'(d,e,r)$, any object $B$ has a presentation $P_1\to P_0\to X\to 0$ in which the $P_i$ are direct sums of tensor products of separable type $F\boxtimes G$ where $F$ and $G$ are finite and the composition factors of $F$ are $p^r$-restricted. Similarly, $B$ has a copresentation $0\to X\to Q^0\to Q^1$ in which the $Q^i$ are products of such tensor products.
\end{lemma}

We can now state the central theorem of this section.

\begin{theorem}\label{thm-eq}
The functor $\Phi$ restricts to an equivalence of categories:
$$\Phi:\StBif\xrightarrow[]{\simeq} \St(d,e,r)\;.$$
\end{theorem}
\begin{proof}
We first prove that $\Phi$ is fully faithful. Let $\mathcal{T}$ be the full subcategory of $\StBif$ supported by the bifunctors of separable type $F\boxtimes G$ with $F$ and $G$ finite.  
By lemma \ref{lm-pres-St2} and exactness of $\Phi$, it suffices to prove that the restriction of $\Phi$ to $\mathcal{T}$ is fully faithful. This follows from corollary \ref{thm-1bis}. 
To prove that 
$\Phi$ is essentially surjective, we consider the functor 
$\Psi: \PP_{d+ep^r,\kk}\to \PP_{d,e,\kk}(2)$
which sends a functor $F$ to the bifunctor 
$$(\Psi F)(V,W)=\hom_{\PP_{d+p^re,\kk}}(\Gamma^{d,V}\otimes (\Gamma^{e,W})^{(r)},F)\;.$$
If $F\otimes G^{(r)}$ is a tensor product of Steinberg type, by corollary \ref{cor-bounded} $F$ is right $p^r$-bounded, so that corollary \ref{thm-1bis} and \cite[Thm 2.10]{FS} yield isomorphisms of strict polynomial functors:
\begin{align*}(\Phi\Psi (F\otimes G^{(r)}))(V)&\simeq \hom_{\PP_{d,\kk}}(\Gamma^{d,V},F)\otimes \hom_{\PP_{e,\kk}}(\Gamma^{e,V^{(r)}},G)\\&\simeq F(V)\otimes G(V^{(r)})\;.\end{align*}
Thus $\Phi\circ\Psi$ is the identity functor on the tensor products of Steinberg type. By lemma \ref{lm-pres-St}, all the functors of $\St(d,e,r)$ are kernels of products of tensor products of Steinberg type. Thus by left exactness of $\Phi\circ\Psi$, the restriction of $\Phi\circ\Psi$ to the whole category $\St(d,e,r)$ is isomorphic to the identity functor. Hence $\Phi$ is essentially surjective (and $\Psi$ is the inverse of $\Phi$). 
\end{proof}

Theorem \ref{thm-eq} generalizes Steinberg tensor product theorem. Indeed, external tensor products $L_\lambda\boxtimes L_\mu$ of simple functors are simple bifunctors, so that theorem \ref{thm-eq} and the stability of $\St(d,e,r)$ by subobjects imply that the functors $\Phi(L_\lambda\boxtimes L_\mu)=L_\lambda\otimes L_\mu^{(r)}$ is simple. More generally, theorem \ref{thm-eq} can be used to convert any question about the structure of the tensor products of Steinberg type (socle length, submodule lattices, or even $\Ext^1$  issuesù) into similar questions about the structure of bifunctors of separable type which are much easier to study. To illustrate this, we give new properties of tensor products of Steinberg type, obtained by translating some general properties of representations of tensor products of finite dimensional algebras given in section \ref{sec-reptens} (recall that the category $\PP_{d,e,\kk}(2)$ is equivalent to the 
category of $S(d,d)\otimes S(e,e)
$-modules). 

\begin{remark}
In the following propositions, we do not assume that $F$ and $G$ are homogeneous. In each case, the proof reduces easily to the homogeneous case by additivity of tensor products. We also observe that each of these propositions is a stronger statement than the classical Steinberg tensor product theorem.
\end{remark}

\begin{corollary}[Socle series]
If the composition factors of $F$ are $p^r$-restricted, then for all $G$, the socle filtration of $F\otimes G^{(r)}$ is the tensor product of the socle filtration of $F$ by the socle filtration of $G$, precomposed by $I^{(r)}$.
\end{corollary}

\begin{corollary}[Subfunctors]\label{prop-subfun}
Assume that the composition factors of $F$ are $p^r$-restricted. Let $G$ be any functor. Assume that $F$ or $G$ is multiplicity free. Then the subfunctors $S\subset F\otimes G^{(r)}$ are of the following form, for some families of subfunctors $F_\alpha\subset F$ and $G_\alpha\subset G$:
$$ S=\sum_{\alpha} F_\alpha\otimes G_\alpha^{(r)}  \;.$$
\end{corollary}

\begin{corollary}[Diagrams]\label{prop-diagrfun}
Assume that $F$ and $G$ are multiplicity-free and the composition factors of $F$ are $p^r$ restricted. Then the diagram associated to $F\otimes G^{(r)}$ as defined in \cite{Alperin} has the functors 
$L_\lambda\otimes L_\mu^{(r)}$ as vertices, where $L_\lambda$ is a composition factor of $F$ and $L_\mu$ is a composition factor of $G$, and there is an edge $L_\lambda\otimes L_\mu^{(r)}\to L_\lambda'\otimes {L_\mu'}^{(r)}$ if and only if one of the following two conditions holds: 
\begin{enumerate}
\item[(i)] $L_\lambda=L'_\lambda$ and there is an edge $L_\mu\to L_\mu'$ in the diagram of $G$,
\item[(ii)] $L_\mu=L'_\mu$ and there is an edge $L_\lambda\to L_\lambda'$ in the diagram of $F$.
\end{enumerate}
\end{corollary}

Next statement follows from proposition \ref{prop-localizing}. It uses  that all simple strict polynomial functors satisfy $\Ext^1_{\PP_\kk}(L,L)=0$, which follows from the fact that the Schur algebras are quasi-hereditary. 

\begin{corollary}[Tensor products on the left]\label{prop-tensleft}
Let $\lambda$ be a $p^r$-restricted partition. Let $L_\lambda\otimes \PP_\kk^{(r)}$ denote the full subcategory of $\PP_\kk$ whose objects are the functors isomorphic to tensor products of the form $L_\lambda\otimes F^{(r)}$. Then:
\begin{enumerate}
\item the subcategory $L_\lambda\otimes \PP_\kk^{(r)}$ is localizing and colocalizing,
\item precomposing by $I^{(r)}$ and tensoring by $L_\lambda$ yields an equivalence of categories $\PP_\kk\simeq L_\lambda\otimes \PP_\kk^{(r)}$.
\end{enumerate}
\end{corollary}

\section{Application to internal tensor products}\label{sec-tensint}

The purpose of this section is to study the internal tensor product of simple functors. In particular theorem \ref{thm-SteinInternal} plays a role for internal tensor products similar to the role of the Steinberg theorem for ordinary tensor products. 

\subsection{Internal tensor products of simple functors}
Let $F_1$ and $G_1$, (resp. $F_2$ and $G_2$) be two homogeneous functor of degree $d$ (resp. $e$).  The internal tensor product is equipped with a coproduct 
$$(F_1\otimes F_2)\uotimes (G_1\otimes G_2) \to (F_1\uotimes G_1)\otimes (F_2\uotimes G_2)\;.$$
To be more specific, this coproduct coincides on the standard projectives with the following composite (where the first map is the canonical inclusion and the second one is the canonical projection):
\begin{align*}(\Gamma^{d,T}\otimes \Gamma^{e,U})\uotimes (\Gamma^{d,V}\otimes \Gamma^{e,W}) &\hookrightarrow (\Gamma^{d+e,T\oplus U}\uotimes \Gamma^{d+e,V\oplus W})\\
&= \Gamma^{d+e, (T\oplus U)\otimes (V\oplus W)}\\
&\twoheadrightarrow \Gamma^{d,T\otimes V}\otimes \Gamma^{e,U\otimes W}\\
&= (\Gamma^{d,T}\uotimes \Gamma^{d,V})\otimes (\Gamma^{e,U}\uotimes \Gamma^{e,W})\;.
\end{align*}
The following proposition is a consequence of corollary \ref{thm-1bis}. 

\begin{proposition}\label{prop-1ter}
Let $F,G,X,Y$ be homogeneous strict polynomial functors, and let $r\ge 0$. If $\deg F< \deg G$ and $G$ is left $p^r$-bounded, or if $\deg F> \deg G$ and $F$ is left $p^r$-bounded, then 
$$(F\otimes X^{(r)})\uotimes (G\otimes Y^{(r)})=0\;.$$
If $\deg F=\deg G$ and $F$ or $G$ is left $p^r$-bounded, then the coproduct induces an isomorphism:
$$(F\otimes X^{(r)})\uotimes (G\otimes Y^{(r)}) \simeq (F\uotimes G)\otimes (X\uotimes Y)^{(r)}\;.$$
\end{proposition}
\begin{proof}
In this proof, we assume that $\deg F\ge \deg G$ and $F$ is $p^r$-bounded (the proof is similar if $\deg F\le \deg G$ and $G$ is $p^r$-bounded). Since the internal tensor product is right exact and commutes with arbitrary direct sums, it suffices to prove proposition \ref{prop-1ter} when $G$ and $Y$ are finite. 

Since $F$ is left $p^r$-bounded, the parameterized functor $F^V$ also is. Hence, if $\deg F>\deg G$, corollary \ref{thm-1bis} implies that 
$$\uHom(F\otimes X^{(r)},G^\sharp\otimes (Y^\sharp)^{(r)})=0\;.\qquad  (*) $$
Since $G$ and $Y$ are finite, $G^\sharp\otimes (Y^\sharp)^{(r)}$ is isomorphic to $(G\otimes Y^{(r)})^\sharp$, hence equality $(*)$ can be reinterpreted as 
$$\left(\,(F\otimes X^{(r)})\uotimes (G\otimes Y^{(r)})\,\right)^\sharp=0\;.$$
This proves the asserted cancellation. Assume now that $\deg F=\deg G$. Then by corollary \ref{thm-1bis} the coproduct induces an isomorphism:
$$\uHom(F,G)\otimes \uHom(X,Y)^{(r)}\simeq \uHom(F\otimes X^{(r)},Y\otimes Y^{(r)})\;.\quad(**) $$
But the coproduct is dual to the cup product, that is, for all functors $F$, $G$, $H$, $K$ there is a commutative diagram, in which the horizontal isomorphisms are the canonical isomorphisms recalled in section \ref{subsec-internal}:
$$ 
\xymatrix{
(F\uotimes H)^\sharp\otimes (G\uotimes K)^\sharp\ar[d]^-{\mathrm{can}}\ar[r]^-{\simeq}&\uHom(F,H^\sharp)\otimes \uHom(G,K^\sharp)\ar[d]^{\cup}\\
(F\uotimes H)\otimes (G\uotimes K))^\sharp \ar[d]^-{\mathrm{coproduct}^\sharp}&
\uHom(F\otimes H,G^\sharp\otimes K^\sharp)\ar[d]^-{\mathrm{can}}\\
(F\otimes H,G\otimes K)^\sharp\ar[r]^-{\simeq}& \uHom(F\otimes H,(G\otimes K)^\sharp)
}\;.
$$
If the functors $G$ and $K$ are finite so is $G\uotimes K$ and the canonical maps denoted `$\mathrm{can}$' in the diagram above are isomorphisms. Thus, the isomorphism of propositition \ref{prop-1ter} can be deduced from the diagram above with $H=X^{(r)}$ and $K=Y^{(r)}$, and from isomorphism $(**)$.
\end{proof}

The following theorem reduces the study of internal tensor products of simple functors to the case of $p$-restricted simple functors. In other terms, it plays the same role for internal tensor products as the classical Steinberg tensor product theorem does for ordinary tensor products.

\begin{theorem}\label{thm-SteinInternal}
Let $\lambda^0,\dots,\lambda^r$ and $\mu^0,\dots,\mu^s$ be $p$-restricted partitions, and let $\lambda=\sum p^i\lambda^i$ and $\mu=\sum p^i\mu^i$. 
\begin{enumerate}
\item If $r=s$ and $\mu^i$ and $\lambda^i$ have the same weight for all $i$, then $L_\lambda\uotimes L_\mu$ is nonzero and there is an isomorphism:
$$L_\lambda\uotimes L_\mu\simeq (L_{\lambda^0}\uotimes L_{\mu^0})\otimes (L_{\lambda^1}\uotimes L_{\mu^1})^{(1)}\otimes \dots \otimes (L_{\lambda^r}\uotimes L_{\mu^r})^{(r)}\;. $$
\item Otherwise, $L_\lambda\uotimes L_\mu$ is zero.
\end{enumerate} 
\end{theorem}
\begin{proof}
The classical Steinberg tensor product theorem shows that $L_\lambda=L_{\lambda^0}\otimes\dots \otimes L_{\lambda^r}^{(r)}$ and $L_\mu=L_{\mu^0}\otimes\dots \otimes L_{\mu^s}^{(s)}$ where the $L_{\lambda^i}$ and the $L_{\mu^j}$ are $p$-restricted, hence right $p$-bounded by corollary \ref{cor-bounded}. Hence the result follows by applying proposition \ref{prop-1ter}.
\end{proof}

\subsection{The case of $p$-restricted simple functors}\label{subsec-intpres}
To investigate internal tensor products of $p$-restricted simple functors, we will rely on the Schur functor.
\begin{lemma}\label{lm-iso-iths}
For all strict polynomial functors $F$, there are isomorphisms of functors, natural with respect to $F$:
$$F\uotimes \otimes^d\simeq \uHom(\otimes^{d},F)\simeq \otimes^d\otimes f_d(F)\;.$$
Moreover, if we consider the action of $\Si_d$ on the left hand side induced by the left action of $\Si_d$ on $\otimes^d$, the action on the middle term induced by the right action of $\Si_d$ on $\otimes^d$, and the diagonal action of $\Si_d$ on the right hand side, then these isomorphisms are $\Si_d$-equivariant.
\end{lemma}
\begin{proof}
We have isomorphisms of strict polynomial functors, natural with respect to $V$ and $F$: 
$$\uHom(\Gamma^{d,V},F)\simeq F_V\simeq F\uotimes \Gamma^d_V\;.\qquad (*)$$
Take $V=\kk^d$ and let the torus $(\mathbb{G}_m)^{\times d}$ act on $\kk^d$ by $(\lambda_1,\dots,\lambda_d)\cdot (x_1,\dots,x_d)=(\lambda_1x_1,\dots,\lambda_dx_d)$. Then the summand of weight $(1,\dots,1)$ of the right hand side of isomorphism $(*)$ is $F\,\uotimes\, \otimes^d$, and it is isomorphic to the summand of weight $(1,\dots,1)$ of the left hand side, which is $\uHom(\otimes^d,F)$. 
Moreover, $\uHom(\otimes^d,F)$ is isomorphic to the functor $U\mapsto \hom_{\PP_\kk}((\otimes^d)^U,F)$. Since $(\otimes^d)^{U}$ is isomorphic to $(U^{\vee})^{\otimes d}\otimes\otimes^d$, we get an isomorphism of strict polynomial functors with variable $U$:
$$\uHom(\otimes^d,F)\simeq \hom_{\PP_\kk}((U^{\vee})^{\otimes d}\otimes\otimes^d,F)\simeq U^{\otimes^d}\otimes f_d(F)\;.$$
Finally, one easily checks that these explicit constructions of the isomorphisms of lemma \ref{lm-iso-iths} yield $\Si_d$-equivariant isomorphisms.
\end{proof}

\begin{proposition}\label{prop-Schur-Kronecker}
For all functors $F$, $G$, there is an isomorphism of $\kk\Si_d$-modules, where the tensor product on the right is the Kronecker product of $f_d(F)$ and $f_d(G)$ (i.e. $\Si_d$ acts diagonally):
$$f_d(F\uotimes G)\simeq f_d(F)\otimes f_d(G)\;.$$ 
\end{proposition}
\begin{proof}
Lemma \ref{lm-iso-iths} yields a chain of isomorphisms:
$$F\uotimes (G\uotimes \otimes^d)\simeq F\uotimes (\otimes^d \otimes f_d(G))\simeq 
(F\uotimes \otimes^d)\otimes f_d(G)\simeq \otimes^d\otimes f_d(F)\otimes f_d(G)\;.$$
Thus the evaluation of $F\uotimes (G\uotimes \otimes^d)$ at $\kk$ is isomorphic to $f_d(F)\otimes f_d(G)$. On the other hand $F\uotimes (G\uotimes \otimes^d)$ is isomorphic to $(F\uotimes G)\uotimes \otimes^d$ and lemma \ref{lm-iso-iths} shows that the evaluation of the latter at $\kk$ is isomorphic to $f_d(F\otimes G)$.
\end{proof}

The following corollary shows that in the first case of theorem \ref{thm-SteinInternal}, the internal tensor product is always nonzero.

\begin{corollary}
Let $L$ and $L'$ be two $p$-restricted simples. Then $L\uotimes L'$ is nonzero.
\end{corollary}
\begin{proof}
By Clausen and James' theorem, $f_d(L)$ and $f_d(L')$ are nonzero. Hence by proposition \ref{prop-Schur-Kronecker}, $f_d(L\uotimes L')$ is nonzero. Thus  $L\uotimes L'$ is nonzero.
\end{proof}

Given two $p$-restricted simples $L$ and $L'$ a natural question is to determine if the analogue of theorem \ref{thm-tensprespres} holds, i.e. if the nonzero functor $L\uotimes L'$ is simple. In fact, Bessenrodt and Kleshchev have proved \cite{BK} that the Kronecker product of two simple representations of symmetric groups is almost never simple. In particular, proposition \ref{prop-Schur-Kronecker} has the following consequence in odd characteristic.
\begin{corollary}\label{cor-intss}
Assume that $p$ is odd. Let $L$ and $L'$ be two $p$-restricted simples such that $f_d(L)$ and $f_d(L')$ both have dimension at least two. Then $L\uotimes L'$ is not simple.
\end{corollary}
\begin{proof}
Since the right adjoint of $f_d$ satisfies $f_d\circ r_d=\Id$, $f_d(L)$ sends simple functors either to simple $\kk\Si_d$-modules or to zero. But $f_d(L\uotimes L')\simeq f_d(L)\otimes f_d(L')$ is a Kronecker product of two simple $\kk\Si_d$-modules, hence is not simple by \cite[Thm 2]{BK}. Thus $L\uotimes L'$ cannot be simple.
\end{proof}

\begin{remark}
Corollary \ref{cor-intss} uses \cite[Thm 2]{BK}, which is a nontrivial result on symmetric groups. It would be interesting to find a more elementary proof of corollary \ref{cor-intss}, in the spirit of the proof of theorem \ref{thm-tensprespres}.
\end{remark}

To solve completely (in odd characteristic) the problem of knowing if an internal tensor product $L\uotimes L'$ can be simple, it remains to study the case where $f_d(L')$ has dimension one. The remainder of the section is devoted to this study. In our discussion below, we show in corollary \ref{cor-calcul} that when $f_d(L')$ has dimension one, $L\uotimes L'$ may sometimes be simple and sometimes not, and in corollary \ref{cor-lesdeuxpareil} we show that it suffices to study the case $L'=Q^d$. The latter case is studied in \cite{Reischuk}, where the simplicity of $L\uotimes Q^d$ is shown to be equivalent to $p(L,1)>1$.

There are two $\kk\Si_d$-modules of dimension $1$, namely the signature module $\kk^{\mathrm{alt}}$ and the trivial module $\kk$. The signature module is the image by the Schur functor of $\Lambda^d=L_{(1,\dots,1)}$. 
Since $\hom_{\PP_{d,\kk}}(\otimes^d,S^d)$ has dimension $1$, and since $S^d$ is a quotient of $\otimes^d$, the head of $S^d$ is a $p$-restricted simple functor. This functor is known under the name of truncated symmetric powers, and we denote it by $Q^d$ as in \cite{BMT}. Then $f_d(Q^d)$ is the trivial $\kk\Si_d$-module.
Thus, to solve completely (in odd characteristic) the problem of knowing if an internal tensor product $L\uotimes L'$ can be simple, it remains to investigate the internal tensor products $L\uotimes Q^d$ and $L\uotimes \Lambda^d$ for $p$-restricted simples $L$. 
\begin{proposition}\label{prop-calcul}
Let $F$ be a homogeneous functor of degree $d$. Consider the right action of $\Si_d$ on $\otimes^d$ given by permuting the factors of the tensor product. If $p\ne 2$ then 
$$F\uotimes \Lambda^d\;\simeq \;(\otimes^d)\otimes_{\Si_d} (\kk^{\mathrm alt}\otimes f_d(F))\;.$$
If $\mathrm{Head}(F)$ is a direct sum of $p$-restricted simples (and $p$ arbitrary), then
$$F\uotimes Q^d\;\simeq \;(\otimes^d)\otimes_{\Si_d} f_d(F)\;.$$ 
\end{proposition}
\begin{proof}
Lemma \ref{lm-iso-iths} yields a $\Si_d$-equivariant isomorphism between $F\uotimes \otimes^d$ and $\otimes^d\otimes f_d(F)$. Taking the coinvariants under the signed action of $\Si_d$, and using right exactness of internal tensor products we obtain the first isomorphism. For the second isomorphism, let $R^d$ be the radical of $S^d$. Since $f_d(S^d)=f_d(Q^d)$ and the Schur functor is exact, we have $f_d(R^d)=0$. Hence by lemma \ref{lm-iso-iths}, $R^d\uotimes \otimes^d$ is zero. But if $P$ is left $p$-bounded projective, it is a direct summands in a direct sums of copies of $\otimes^d$, hence $R^d\uotimes P$ is zero. Now $F$ is left $p$-bounded by corollary \ref{cor-bounded}, hence $R^d\otimes F=F\uotimes R^d=0$. By right exactness of tensor products we thus obtain an isomorphism $F\uotimes S^d\simeq F\uotimes Q^d$. Then the computation of $F\uotimes S^d$ is done in the same fashion as the one of $F\uotimes\Lambda^d$.
\end{proof}

If $M$ is a simple $\Si_d$-module, then $M\otimes\kk^{\mathrm{alt}}$ is also simple. Let $L_\mu$ be the simple $p$-restricted functor such that $f_d(L_\mu)=M$. We denote by $m(\mu)$ the $p$-restricted partition such that $f_d(L_{m(\mu)})=M\otimes\kk^\mathrm{alt}$. The involution $\mu\mapsto m(\mu)$ (or rather $\mu'\mapsto m(\mu')$ where $\mu'$ stands for the conjugate partition of $\mu$) is known as the Mullineux correspondence \cite[Chap. 4.2]{Martin}, and its combinatorial description has been proved by Ford and Kleshchev \cite{FK}, see also the work of Brundan and Kujawa \cite{BrunK} for a more recent and different proof. Proposition \ref{prop-calcul} has the following consequence.
\begin{corollary}\label{cor-lesdeuxpareil}
Let $\mu$ be a $p$-restricted partition. Then $$L_\mu\uotimes \Lambda^d\simeq L_{m(\mu)}\uotimes Q^d\;.$$
\end{corollary}

As another consequence of proposition \ref{prop-calcul}, we obtain that the internal tensor product of two simple functors may sometimes be simple and sometimes not. The problem of knowing exactly for which $p$-restricted partitions $\mu$ the functor $L_\mu\uotimes \Lambda^d$ is simple is studied in \cite{Reischuk}.
\begin{corollary}\label{cor-calcul}Assume that $p$ is odd.
Then $Q^d\uotimes \Lambda^d$ is isomorphic to $\Lambda^d$, and $\Lambda^d\uotimes \Lambda^d$ is isomorphic to $S^d$.
\end{corollary}

\section{Estimates for $p(F,r)$ and $i(F,r)$}\label{sec-estimates}

\subsection{Basic properties of $p(F,r)$ and $i(F,r)$}

Let $r$ be a positive integer. We introduce the following two homogeneous functors of degree $d$, where $T^{(d_0,\dots,d_k)}=(\otimes^{d_0})\otimes(\otimes^{d_1})^{(1)}\otimes\dots \otimes (\otimes^{d_k})^{(k)}$ as in corollary \ref{cor-caract}:
\begin{align*}
&L(d,r)=\bigoplus_{\text{$\lambda$ not $p^r$-restricted, and $|\lambda|=d$}}L_\lambda\;, \\
& T(d,r)=\bigoplus_{\sum_{0\le i<r}p^id_i<d\;,\; \sum_{0\le i}p^id_i=d} T^{(d_0,\dots,d_k)}\;.
\end{align*}
These functors are defined so that they contain all the simples of degree $d$, or all the twisted tensor powers of degree $d$, which have at least one factor precomposed by $I^{(s)}$ with $s\ge r$.
Hence they are non zero if and only if $d\ge p^r$. By corollary\ref{cor-caract}, $L(d,r)$ is a quotient of $T(d,r)$. Since these two functors are self-dual, it follows that $L(d,r)$ is also a subfunctor of $T(d,r)$.
\begin{proposition}\label{prop-crithigh} Let $F$ be a homogeneous functor of degree $d$, and let $G(d,r)$ be either equal to 
$L(d,r)$ or to $T(d,r)$.  Then $p(F,r)$, resp. $i(F,r)$,  is the lowest (possibly $+\infty$) degree $k$ such that the vector space $\Ext^k_{\PP_\kk}(F,G(d,r))$, resp. $\Ext^k_{\PP_\kk}(G(d,r),F)$, is nonzero.
\end{proposition}
\begin{proof}
Let $P$ be a degree $d$ homogeneous $p^r$-bounded projective. Then by theorem \ref{thm-1} $\Ext^*_{\PP_\kk}(P,G(d,r))$ is zero.
Take a projective resolution $Q$ of $F$ whose first $p(F,r)$-terms (i.e. up to index $p(F,r)-1$) are left $p^r$-bounded projectives, and let $K$ be the kernel of $Q_{p(F,r)-1}\to Q_{p(F,r)-2}$. By definition of $p(F,r)$, $K$ is not $p^r$-bounded. By corollary \ref{cor-bounded}, this means that there exists a nonzero map $K\to L(d,r)$, hence also a nonzero map $K\to T(d,r)$. By dimension shifting: 
$$\Ext^{i}_{\PP_\kk}(F,G(d,r))\simeq \begin{cases} 0 & \text{ if $i<p(F,r)$}\\
\hom_{\PP_\kk}(K,G(d,r))\ne 0 & \text{ if $i=p(F,r)$}
\end{cases}\;.$$
The proof for $i(F,r)$ is similar.
\end{proof}

Since $T(d,r)$ is a self dual functor, $\Ext^*_{\PP_\kk}(T(d,r),F^\sharp)$ is always isomorphic to $\Ext^*_{\PP_\kk}(F,T(d,r))$. Thus we obtain the following corollary.

\begin{corollary}
For all functors $F$, we have $i(F^\sharp,r)=p(F,r)$
\end{corollary}

We now indicate how $i(F,r)$ behaves with respect to some usual operations on strict polynomial functors. There are similar statements for $p(F,r)$ which can be deduced from the formula $p(F,r)=i(F^\sharp, r)$ or by repeating the proofs with projective resolutions. We leave this to the reader.
\begin{proposition}\label{prop-op} Let $F$ and $G$ be two functors. The following holds:
\begin{align*}
(a)\quad&i(F_V,r)=i(F,r)\;,\\
(b)\quad&i(F,r)=i(F^{(s)},r+s)\;,\\
(c)\quad&i(F\otimes G,r)=\min\{\,i(F,r)\,,\,i(G,r)\,\}\;,\\
(d)\quad&i(F\oplus G,r)=\min\{\,i(F,r)\,,\,i(G,r)\,\}\;,\\
(e)\quad&i(F,r)\ge \min\{ i(S,r)\;:\; S\text{ is finite and } S\subset F\}\;.
\end{align*}
\end{proposition}
\begin{proof}
Statement $(d)$ is straightforward from the definition of $i(F,r)$. Statement $(d)$ implies that for the remaining statements, we can assume that $F$ and $G$ are homogeneous. We let $d:=\deg F$ and $g:=\deg G$. Statement $(e)$ follows from the interpretation of $i(F,r)$ given in proposition \ref{prop-crithigh} and the fact that $\Ext^*(T(d,r),-)$ commutes with directed colimits. To prove $(a)$, observe that $F$ is a direct summand in $F^V$ so that $i(F,r)\ge i(F_V,r)$. Moreover, if $J$ is a standard $p^r$-bounded injective then $J_V$ is a direct sum of standard $p^r$-bounded injectives. Hence if $Q$ is an injective resolution of $F$ whose first $i(F,r)$ terms are left $p^r$-bounded injectives, then $Q_V$ is an injective resolution of $F_V$ whose first $i(F,r)$ terms are left $p^r$-bounded injectives. So that $i(F_V,r)\ge i(F,r)$. To prove $(b)$, we use the isomorphisms 
$$\Ext^*_{\PP,\kk}(T(d,r+s),F^{(s)})\simeq \Ext^*_{\PP_\kk}(T(d,r)^{(s)},F^{(s)})\simeq \Ext^*_{\PP_\kk}(T(d,r),F_{E_s})\;.$$
The first isomorphism is induced by the inclusion $T(d,r)^{(s)}\subset T(d,r+s)$, the cokernel of this split inclusion is easily seen to be zero by using the sum diagonal adjunction. The second isomorphism is proved in \cite{TouzeClasses,Chalupnik15}.
In this formula $F_{E_s}$ is a nonnegatively graded functor, and the degree on the right hand side is the total degree. The graded functor $F_{E_s}$ equals $F_{\kk^{p^s}}$ in an ungraded way, so that the lowest nonzero degree $k$ on the right hand side of the isomorphism is greater or equal to $i(F_{\kk^{p^s}},r)=i(F,r)$. Conversely, the degree zero component of $F_{E_s}$ is isomorphic to $F$ so that the the lowest nonzero degree $k$ on the right hand side of the isomorphism is greater or equal to $i(F,r)$. It remains to prove $(c)$. Assume for example that $i(F,r)\le i(G,r)$. If $Q$ resp. $Q'$ are two injective resolutions of $F$, resp. $G$, whose first $i(F,r)$ terms are $p^r$-bounded, then $Q\otimes Q'$ is an injective resolution of $F\otimes G$ whose first $i(F,r)$ terms are $p^r$ bounded, hence $i(F\otimes G,r)\ge i(F,r)$. Conversely, let $x$ be a nonzero extension in $\Ext^{i(F,r)}(T(d,r),F)$. By corolalry \ref{cor-caracteris} we can find a tuple $(d_0,\dots,d_\ell)$ and a nonzero map $f:T^{(d_0,\dots,d_\ell)}\to G$. Since cup products are injective (by theorem \ref{thm-1} with $r=0$ or by lemma \ref{lm-cup}) $x\cup f$ is a nonzero element of $\Ext^{i(F,r)}(T(d,r)\otimes T^{(d_0,\dots,d_\ell)},F\otimes G)$. But $T(d,r)\otimes T^{(d_0,\dots,d_\ell)}$ is a direct summand of $T(d+g,r)$, so that $i(F\otimes G,r)\le i(F,r)$.
\end{proof}

\subsection{A few examples}\label{subsec-ex}

\begin{proposition}\label{prop-ex} Let $r$ be a nonnegative integer. The following holds.
\begin{enumerate}
\item If $\deg F< p^r$ then $i(F,r)=+\infty$.
\item If $d\ge p^r$, then $i(S^d,r)=0$.
\item If $d\ge p^r$, then $i(\Lambda^d,r)=p^r-1$.
\item If $d\ge p^r$, then $i(\Gamma^d,r)=2(p^r-1)$.
\end{enumerate}
\end{proposition}
\begin{proof}
The first statement follows from the fact that when $d<p^r$, all basic injectives of degree $d$ are $p^r$-bounded. If $d\ge p^r$ the multiplication of the symmetric algebra and the natural inclusion $I^{(r)}\hookrightarrow S^{p^r}$ induce a nonzero map $\otimes^{d-p^r}\otimes I^{(r)}\to S^{d}$. Hence by proposition \ref{prop-crithigh} $i(S^d,r)=0$. Let us prove that $i(\Lambda^d,p^r)=p^r-1$. The homogeneous part of degree $d$ of the reduced bar construction of the symmetric algebra $S$ provides an injective resolution of $\Lambda^d$ whose first $p^r-1$ terms are basic $p^r$-bounded injectives (see e.g \cite{Totaro}). Thus $i(\Lambda^d,r)\ge p^r-1$. Conversely, using sum-diagonal adjunction one obtains that 
$\Ext^*_{\PP_\kk}(\otimes^{d-p^r}\otimes I^{(r)},\Lambda^d)$ is isomorphic to the tensor product 
$$\Ext^*_{\PP_\kk}(\otimes^{d-p^r},\Lambda^{d-p^r})\otimes \Ext^*_{\PP_\kk}(I^{(r)},\Lambda^{p^r})\;.$$
The factor on the left hand side of the tensor product is one dimensional and concentrated in degree zero, and by \cite[(4.5.1) p. 251]{FS}, the factor on right hand side of the tensor product is one dimensional and concentrated in degree $p^r-1$. Now $\otimes^{d-p^r}\otimes I^{(r)}$ is a direct summand in $T(d,r)$ so that $i(\Lambda^d,r)\le p^r-1$ by proposition \ref{prop-crithigh}. A similar argument gives the case of $\Gamma^d$: the homogeneous part of degree $d$ of twofold reduced bar construction of the symmetric algebra yields anjective resolution of first $2(p^r-1)$ terms are basic $p^r$-bounded injectives, and on the other hand one can compute that $\Ext^*_{\PP_\kk}(\otimes^{d-p^r}\otimes I^{(r)},\Gamma^d)$ is one dimensional and concentrated in degree $2(p^r-1)$.
\end{proof}

Let us denote by $S_\lambda$ the Schur functor associated to a partition $\lambda$ and by $W_\lambda$ the Weyl functor associated to $\lambda$. These are finite homogeneous strict polynomial functors, whose degree is the weight of the partition $\lambda$, and we have $W_\lambda=S_\lambda^\sharp$. They generalize the functors $S^d$, $\Lambda^d$ and $\Gamma^d$, indeed:
$$W_{(d,0,0,\dots)}=\Gamma^d,\quad S_{(d,0,0,\dots)}=S^d,\quad S_{(1,\dots,1)}=W_{(1,\dots,1)}=\Lambda^d\;.$$
The $S_\lambda$ (resp. $W_\lambda$) are the costandard (resp. standard) objects of the highest weight category structure of $\PP_\kk$. We refer the reader to \cite[section 6.1.1]{TouzeRingel}, or \cite{Krause2} for more details and references on these functors. The following lemma may be useful for computations.
\begin{lemma}\label{lm-ptitlm}Let $\lambda$ be a partition and let $\lambda'$ be the dual partition. for all tuples $(d_0,\dots,d_k)$ there is a graded isomorphism (where $\Ext$ are understood as zero in negative degrees):
$$\Ext^{*}_{\PP_\kk}(T^{(d_0,\dots,d_k)},S_\lambda)\simeq \Ext^{*+s}_{\PP_\kk}(T^{(d_0,\dots,d_k)},W_{\lambda'})\,$$
where $s=d_1(p-1)+d_2(p^2-1)+\dots +d_k(p^k-1)$.
\end{lemma}
\begin{proof}
We use Ringel duality $\Theta$, which is an autoequivalence of $\DD(\PP_{d,\kk})$.
See e.g. \cite[Section 3]{TouzeRingel}, \cite[Section 2]{Chalupnik2}. We have $\Theta S_\lambda=W_{\lambda'}$ and $\Theta T^{(d_0,\dots,d_k)}=T^{(d_0,\dots,d_k)}[-s]$, so that the lemma follows from interpreting morphisms of degree $s$ in the derived category as extensions of degree $s$.
\end{proof}

\begin{proposition}
Let $\lambda$ be a partition, let $\lambda'$ be the dual partition. Then we have: 
$$ i(S_\lambda,r)+p^r-1\le i(W_{\lambda'},r)\;.$$
Assume moreover that $\lambda=\lambda^0+p\lambda^1+\dots +p^k\lambda^k$ where each $\lambda^k$ is a $p$-restricted partition of $d_k$, and $k\ge r$. Then we have:
$$i(W_{\lambda'},r)\le \sum_{i=1}^k d_i (p^i-1)\;.$$
\end{proposition}
\begin{proof}
We use the isomorphism of lemma \ref{lm-ptitlm}. If $T^{(d_0,\dots,d_r)}$ is a direct summand of $T(d,r)$, then the associated shift $s$ is always greater or equal to $p^r-1$. This proves the first inequality. Now corollary \ref{cor-caract} implies that 
$\hom_{\PP_\kk}(T^{(d_0,\dots,d_k)}, S_\lambda)\ne 0$. Hence there is a nonzero extension of degree $\sum_{i=1}^k d_i (p^i-1)$ between $\otimes^{d-ep^r}\otimes (\otimes^{e})^{(r)}$ and $W_{\lambda'}$, which proves the second inequality.
\end{proof}

We finish this section by computating the integers $i(F,r)$ when $F$ is any Schur or Weyl functor of degree $4$ in characteristic $p=2$. The result is already known for $S^4$, $\Lambda^4$ and $\Gamma^4$ by proposition \ref{prop-ex}. For the three remaining partitions, the computation relies on the following short exact sequences.
\begin{lemma}\label{lm-ses} Let $\kk$ be a field of characteristic $p=2$.
There are short exact sequences:
\begin{align}
& 0\to \Lambda^4\to \Lambda^3\otimes\Lambda^1\to S_{(2,1,1)}\to 0\\
& 0\to \to S_{(3,1)}\to S^3\otimes S^1\to S^4\to 0\\
& 0\to S_{(2,2)}\to S^2\otimes S^2\to S_{(3,1)}\;\oplus\; S^4\to 0
\end{align}
\end{lemma}
\begin{proof}
The first two sequences are the standard presentation and copresentation of Schur functors and are valid over an arbitrary ring \cite{ABW}, only the last one is specific to the characteristic two case and needs to be proved. As proved in \cite{ABW}, the Schur functor $S_{(2,2)}$ has copresentation given by:
$$0\to S_{(2,2)}\to S^2\otimes S^2\xrightarrow[]{(\phi,\mathrm{mult})}S^3\otimes S^1\;\oplus \;S^4\qquad (*)$$
where $\mathrm{mult}$ denotes the map induced by the multiplication for the symmetric algebra and $\phi$ is defined as the following composite map, where $\Delta$ is induced by the comultiplication of the symmetric algebra.
$$S^2\otimes S^2\xrightarrow[]{S^2\otimes \Delta} S^2\otimes \otimes^2\xrightarrow[]{\mathrm{mult}\otimes S^1} S^3\otimes S^1\;.$$
Since the field has characteristic $2$, there is a surjective map $\pi:S^2\to \Lambda^2$, and $\phi$ factors in a unique way as $\phi=\psi\circ (S^2\otimes\pi)$. Now the composite $\mathrm{mult}\circ\psi: S^2\otimes \Lambda^2\to S^4$ is zero, so that the image of $\psi$ is contained into $S_{(3,1)}$. Thus the copresentation $(*)$ induces a copresentation:
$$0\to S_{(2,2)}\to S^2\otimes S^2\to S_{(3,1)}\;\oplus\; S^4\;.$$
Then the last map on the right is surjective for Euler characteristic reasons (the dimensions being independent of the characteristic, one can do the computation in characteristic zero, where $S^2\otimes S^2$ is isomorphic to $S_{(2,2)}\oplus S_{(3,1)}\oplus S^4$ by the Pieri rule).
\end{proof}

The extension groups between $\otimes^2\otimes I^{(1)}$, $I^{(1)}\otimes I^{(1)}$ or $I^{(2)}$ on the one hand, and tensor products of symmetric or exterior powers on the other hand are easy to compute. Now one can completely compute the extension groups between $T(4,r)$ and the Schur functors simply by inspecting the (not very) long exact $\Ext^*_{\PP_\kk}(T(d,r),-)$-sequences associated to the short exact sequences of lemma \ref{lm-ses}. And one can then obtain the corresponding computations with Weyl functors by lemma \ref{lm-ptitlm}. We record the resulting computations of $i(F,r)$ in the following proposition. Since $p^3=2^3>4=d$, only the cases $r=1$ and $r=2$ are interesting.
\begin{proposition}
Let $\kk$ be a field of characteristic $2$. The following computations hold.
$$\begin{array}{c|cccccccccc}
F &\Gamma^4 & W_{(3,1)} & W_{(2,2)} & W_{(2,1,1)} & \Lambda^4 & S_{(2,1,1)} & S_{(2,2)} & S_{(3,1)} & S^4\\\hline
i(F,1) & 2 & 2 & 2 & 1 & 1 & 1 & 0 & 0 & 0 \\
i(F,2) & 6 & 5 & 4 & 4 & 2 & 2 & 1 & 1 & 0
\end{array}
$$
\end{proposition}

\begin{remark}
One sees on this example that the integers $i(F,r)$ are increasing with respect to the dominance order for Weyl functors, and decreasing with respect to the dominance order for Schur functors. It would be quite interesting to know if this is the shadow of some general phenomenon.
\end{remark}

\section{Application to symmetric groups}\label{sec-applicsym}

\begin{lemma}\label{lm-lm}
The Schur functor sends $p$-bounded projectives (resp. injectives) to projective (resp. injective) $\kk\Si_d$-modules. Moreover, if $F$ is a $p$-bounded projective or if $G$ is a $p$-bounded injective, the Schur functor induces an isomorphism:
$$\hom_{\PP_{d,\kk}}(F,G)\xrightarrow[]{\simeq} \hom_{\Si_d}(f_d(F),f_d(G))\;.$$ 
\end{lemma}
\begin{proof}
The left adjoint of $f_d$ sends $\kk\Si_d$ to $\otimes^d$. Thus $f_d(\otimes^d)=f_d(\ell_d(\kk\Si_d))\simeq \kk\Si_d$ is projective. Moreover, the map induced by $f_d$
\begin{align*}
\hom_{\PP_{d,\kk}}(\otimes^d,G)\to \hom_{\Si_d}(\kk\Si_d,f_d(G))
\end{align*}
is an isomorphism because it identifies with the adjunction isomorphism for $(\ell_d,f_d)$. This proves lemma \ref{lm-lm} for the $p$-bounded projective $\otimes^d$. Now, $f_d$ commutes with arbitrary direct sums, and $p$-bounded projectives are direct summands of direct sums of copies of $\otimes^d$, so this implies that lemma \ref{lm-lm} holds for all $p$-bounded projectives. The proof for $p$-bounded injectives is similar, using the right adjoint $r_d$.  
\end{proof}

The next theorem generalizes many theorems in \cite{KN}. In particular, theorem \ref{thm-KN} does not require any restriction on the characteristic, and works for all $F$ and all $G$. As regards concrete computations, the explicit bounds for $i(G,1)$ for Weyl functors $G$ given section \ref{subsec-ex} yield connectedness bounds which are at least as good as the ones given in \cite{KN}. However, we have not investigated estimates for $i(G,1)$ when $G$ is simple, hence, unlike \cite{KN}, we don't have concrete connectedness estimates for simple functors. 
\begin{theorem}\label{thm-KN}
Let $F$ and $G$ homogeneous functor of degree $d$. The map induced by the Schur functor:
$$\Ext^k_{\PP_{d,\kk}}(F,G)\xrightarrow[]{\simeq} \Ext^k_{\kk\Si_d}(f_d(F),f_d(G))$$
is an isomorphism in degrees $k<p(F,1)+i(G,1)-1$, and it is injective in degree $k=p(F,1)+i(G,1)-1$.
\end{theorem}
\begin{proof}
Assume that there is a short exact sequence $0\to H\to J\to H'\to 0$ where $J$ is a $p$-bounded injective. The Schur functor induces a morphism from the induced $\Ext^*_{\PP_{d,\kk}}(F,-)$-long exact sequence and the induced $\Ext^*_{\Si_d}(f_d(F),f_d(-))$-long exact sequence. Using lemma \ref{lm-lm} together with the five lemma, we see that the Schur functor is $k$-connected for the pair $(F,H)$ (i.e. an isomorphism in $\Ext$-degree $<k$ and injective in $\Ext$-degree $k$) if and only if it is  $(k-1)$-connected for the pair $(F,H')$. Using this argument, we reduce the proof of theorem \ref{thm-KN} to the case where $i(G,1)=0$.
By a similar argument applied to the contravariant variable of $\Ext$, we reduce the proof further to the case where $i(F,1)=1$. In the latter case, $F$ is a quotient of a $p$-bounded projective $P$ and we have a commutative diagram:
$$\xymatrix{
\hom_{\PP_{d,\kk}}(P,G)\ar[r]^-{f_d}_-{\simeq}&\hom_{\Si_d}(f_d(P),f_d(G)) \\
\hom_{\PP_{d,\kk}}(F,G)\ar@{^{(}->}[u]\ar[r]^-{f_d}&\hom_{\Si_d}(f_d(F),f_d(G))\ar@{^{(}->}[u]
}$$
which proves that for the pair $(F,G)$ the Schur functor is indeed $p(F,1)-1$ connected (i.e. injective in degree zero).
\end{proof}

The following examples show that the bounds in theorem \ref{thm-KN} are optimal. 
\begin{example}
Let $Q^p$ be the socle of $\Gamma^p$. Then $Q^p$ is the simple functor with highest weight $(p-1,1)$. In particular $i(Q^p,1)\ge 1$ by corollary \ref{cor-bounded}. Since $\Gamma^p$ is the middle term of a non-split extension:
$$0\to Q^p\to \Gamma^p\to I^{(1)}\to 0\;.$$
This proves that $i(Q^p,1)\le 1$ by proposition \ref{prop-crithigh}, and that $\hom_{\PP_\kk}(\Gamma^p,Q^p)=0$. But $\hom_{\kk\Si_p}(f_p(\Gamma^p),f_p(Q^p))=\kk$ since $f_p(\Gamma^p)=f_p(Q^p)=\kk$. Hence the following map is not an isomorphism
$$\Ext_{\PP_{p,\kk}}^{i(Q^p,1)-1}\left(\Gamma^p,Q^p\right)\to \Ext_{\kk\Si_p}^{i(Q^p,1)-1}\left(f_p(Q^p),f_p(\Gamma^p)\right)\;.$$
\end{example}

\begin{example}
Let $F$ be a homogeneous functor of degree $d$. By proposition \ref{prop-crithigh}, 
$\Ext^{i(F,1)}_{\PP_{\kk,d}}(T(d,1),F)$ is nonzero . On the other hand $f_d(T(d,1))=0$ by corollary \ref{thm-1bis}, so that the following  map is not injective
$$\Ext^{{i(F,1)}}_{\PP_{d,\kk}}\left(T(d,1),F\right)\to \Ext^{{i(F,1)}}_{\kk\Si_d}\left(f_d(T(d,1)),f_d(F)\right)\;.$$
\end{example}

\appendix
\section{Representations of tensor products algebras}\label{sec-reptens}

This appendix collects some results about representations of tensor product algebras. All these results are standard (except maybe proposition \ref{prop-submodules}), but they are scattered in the literature and not always stated under the form that we want to use. 

In the remainder of the section, we fix two finite dimensional algebras $A$ and $B$ over a ground field $\kk$. We assume furthermore that $\kk$ is a splitting field for these two algebras, that is the endomorphism rings of simple modules have dimension one over $\kk$. (This hypothesis is satisfied for quasi-hereditary algebras, and of course for all algebras if $\kk$ is algebraically closed). 

If $M$ is an $A$-module and $N$ is a $B$-module, we denote by $M\boxtimes N$ their tensor product, viewed as an $A\otimes B$-module. Tensor product yields a K\"unneth morphism
$$\Ext^*_{A}(M,M')\otimes \Ext^*_B(N,N')\xrightarrow[]{\kappa} \Ext^*_{A\otimes B}(M\boxtimes M',N\boxtimes N')\;.$$
\begin{proposition}
The K\"unneth morphism $\kappa$ is an isomorphism if $M$ and $M'$ have finite dimension or if $M$ and $N$ have finite dimension.
\end{proposition}
\begin{proof}
If $M$ has finite dimension, then it has a projective resolution by finite dimensional projective $A$-modules. Thus, it suffices to prove the result in degree zero (i.e. for $\hom$), the general result follows formally by taking resolutions. Using semi-exactness and additivity of $\hom$ and $\boxtimes$ with respect to their first variable, one reduces furthermore to the case $M=A$. If $M'$ has finite dimension, the K\"unneth morphism in degree zero identifies with the map 
$M'\otimes  \hom_B(N,N')\to  \hom_B(N,N'\otimes M') $
which is an isomorphism since $M'$ has finite dimension. If $N$ has finite dimension, one can also reduce to the case $N=B$, and in the latter case it is clear that $\kappa$ is an isomorphism.
\end{proof}

\begin{proposition}\label{prop-simple}
Up to isomorphism, the simple $A\otimes B$-modules are the tensor products $L_1\boxtimes L_2$ where $L_1$ (resp. $L_2$) is a simple module over $A$ (resp. $B$). Moreover, two such simple modules $L_1\boxtimes L_2$ and $L_1'\boxtimes L_2'$ are isomorphic if and only if $L_1\simeq L'_1$ and $L_2\simeq L'_2$.
\end{proposition}
\begin{proof}
The fact that $L_1\boxtimes L_2$ is simple if $L_1$ and $L_2$ are simple follows from the density theorem \cite[3.27]{CR}. If two such tensor products $L_1\boxtimes L_2$ and $L_1'\boxtimes L_2'$ are isomorphic, then $L_1\simeq L'_1$ and $L_2\simeq L'_2$ because $\hom_{A\otimes B}(L_1\boxtimes L_2,L_1'\boxtimes L_2')$ is isomorphic to $\hom_A(L_1,L_1')\otimes \hom_B(L_2,L_2')$. It remains to prove that any simple $A\otimes B$-module is of the form $L_1\boxtimes L_2$.
The Jacobson radical $J(A\otimes B)$ of $A\otimes B$ contains $J(A)\otimes B+A\otimes J(B)$ since the latter is a nilpotent ideal \cite[5.15]{CR}. Thus we have a surjective morphism
$$\pi:A/J(A)\otimes B/J(B)\twoheadrightarrow A\otimes B/J(A\otimes B)\;.$$
Since the quotient $C/J(C)$ of a $\kk$-algebra $C$ is a  semisimple ring \cite[5.19]{CR} with the same simple modules as $C$, it follows form the Wedderburn theorem and dimension counting that $\pi$ is an isomorphism and that all simple $A\otimes B$-modules have the form $L_1\boxtimes L_2$.  
\end{proof}

\begin{lemma}\label{lm-socle}
For all modules $M$ and $N$, $\Soc(M)\boxtimes \Soc(N)=\Soc(M\boxtimes N)$.
\end{lemma}
\begin{proof}
By proposition \ref{prop-simple}, $\mathrm{Soc}(M)\boxtimes \mathrm{Soc}(N)$ is a semisimple submodule of $M\boxtimes N$. Moreover, for all simple modules $L_1\boxtimes L_2$, we have:
\begin{align*}
\hom_{A\otimes B}(L_1\boxtimes L_2,M\boxtimes N)&=\hom_A(L_1,M)\otimes \hom_B(L_2,N)\\&=\hom_A(L_1,\mathrm{Soc}(M))\otimes \hom_B(L_2,\mathrm{Soc}(N))\\&= \hom_{A\otimes B}(L_1\boxtimes L_2,\mathrm{Soc}(M)\boxtimes \mathrm{Soc}(N))\;.
\end{align*}
As a consequence, all the simple submodules of $M\boxtimes N$ are submodules of $\Soc(M)\boxtimes \Soc(N)$. This proves the lemma.
\end{proof}

\begin{lemma}\label{lm-head}
For all modules $M$ and $N$, $\Head(M\boxtimes N)=\Head(M)\boxtimes \Head(N)$
\end{lemma}
\begin{proof}
If $M$ and $N$ have finite dimension, the proof is dual to the proof of lemma \ref{lm-socle}. By additivity of $\boxtimes$ with respect to both variables, the result is then true when $M$ and $N$ are arbitrary projectives. In general let $P$, resp. $Q$ be a projective cover of $\mathrm{Head}(M)$, resp. $\mathrm{Head}(N)$. One has quotient maps $P\boxtimes Q\twoheadrightarrow M\boxtimes N\twoheadrightarrow \mathrm{Head}(M)\boxtimes \mathrm{Head}(N)$, and the result follows by taking heads of these modules.
\end{proof}

Lemma \ref{lm-socle} can be applied iteratively to identify the socle filtration of $M\boxtimes N$.  We index socle filtrations of modules so that the $(-1)$-th term is zero and the $0$-th term is the socle of the modules.

\begin{proposition}\label{prop-socle}
For all modules $M$, $N$, the socle filtration of $M\boxtimes N$ is the tensor product of the socle filtration of $M$ with the socle filtration of $N$.
\end{proposition}
\begin{proof} Let $M^i$, $N^i$, $(M\boxtimes N)^i$ be the terms of the socle filtrations of $M$,  $N$, $M\boxtimes N$, and let $F^n:=\sum_{i+j=n} M^i\boxtimes N^j$. We prove $F^n=(M\boxtimes N)^n$ by induction on $n$. We have $F^0=M^0\boxtimes N^0=(M\boxtimes N)^0$ by lemma \ref{lm-socle}. Assume that $F^n= (M\boxtimes N)^n$. Let $\iota$ be the canonical inclusion:
$$\bigoplus_{
i+j=n+1
} (M^{i}/M^{i-1})\boxtimes (N^{j}/N^{j-1}) = F^{n+1}/F^n\hookrightarrow (M\boxtimes N)/F^n\;. $$
Let $\phi$ denote the canonical inclusion 
$$(M\boxtimes N)/F^n\hookrightarrow \bigoplus_{i+j=n+1}(M/M^{j-1})\boxtimes (N/N^{j-1})\;.$$
The composite $\phi\circ\iota$ is the direct sum of the canonical inclusions 
$$(M^{i+1}/M^i)\boxtimes (N^{j+1}/N^j)\hookrightarrow  (M/M^i)\boxtimes (N/N^j)\;.$$ 
Thus, it follows from lemma \ref{lm-socle} that $\phi\circ \iota$ maps the semisimple module $F^{n+1}/F^n$ isomorphically onto the socle of the target of $\phi$. In particular, $\iota$ is an isomorphism.
\end{proof}

Recall that a finite module is multiplicity free if it has a composition series whose composition factors are pairwise non-isomorphic.

\begin{proposition}\label{prop-submodules}
Assume that one of the modules $M$, $N$ is multiplicity free. Then for all submodules $S\subset M\boxtimes N$, there are submodules $U_\alpha$ of $M$ and submodules $V_\alpha$ of $N$ such that 
$S=\sum U_\alpha\boxtimes V_\alpha$.
\end{proposition}
\begin{proof}
Since any module over a finite dimensional algebra is the sum of its finite submodules, it suffices to prove proposition \ref{prop-submodules} when all modules have finite dimension.
Assume for example that $M$ is multiplicity free, and fix a submodule $S\subset M\boxtimes N$.

Let $T$ be a submodule of $S$ such that $T/\mathrm{Rad}(T)\simeq L_1\boxtimes L_2$ is simple.
There is a submodule $U\subset M$ such that $\Head(U)\simeq L_1$. We claim that $T\subset U\boxtimes N$. Indeed, since $M$ is multiplicity free, $L_1\boxtimes L_2$ is not a composition factor of $(M/U)\boxtimes N$. Since $\Head(T)=L_1\boxtimes L_2$, no nontrivial homomorphic image of $T$ can be contained in $(M/U)\boxtimes N$. Thus $T\subset U\boxtimes N$.

We are now going to construct a strictly decreasing sequence of modules $V_0=N\supset V_1\supset \dots \supset V_n$ such that $U\boxtimes V_n=T$. Assume that $V_i$ is constructed, such that $T\subset U\boxtimes V_i$. If the inclusion is an equality then the construction is finished. Otherwise, the canonical map $\phi:\Head(T)\to \Head(U\boxtimes V_i)$ is not sujective. By lemma \ref{lm-head}, $\Head(U\boxtimes V_i)=\Head(U)\boxtimes \Head(V_i)$ and by using the K\"unneth formula, we see that all submodules of $\Head(U\boxtimes V_i)$ are of the form $\Head(U)\boxtimes W$ where $W$ is a submodule of $\Head(V_i)$. In particular the image of $\phi$ is of the form $\mathrm{Im}\,\phi=\Head(U)\boxtimes W_{\phi}$. 
The inverse image of $\mathrm{Im}\,\phi$ by the quotient map $\pi\boxtimes \pi_i: U\boxtimes V_i\twoheadrightarrow \Head(U)\boxtimes \Head(V_i)$ is $\mathrm{Rad}(U)\boxtimes V_i+U\boxtimes \pi_i^{-1}(W_\phi)$. This submodule of $U\boxtimes V_i$ which contains $T$. But $\Head(T)\simeq L_1\boxtimes L_2$ is not a composition factor of
$$\mathrm{Rad}(U)\boxtimes (V_i/\pi_i^{-1}(W_\phi))\simeq \frac{\mathrm{Rad}(U)\boxtimes V_i + U\boxtimes \pi_i^{-1}(W_\phi)}{U\boxtimes \pi_i^{-1}(W_\phi)}\;.$$
Thus $T$ is actually a submodule of $U\boxtimes \pi_i^{-1}(W_\phi)$. We define $V_{i+1}:=\pi_i^{-1}(W_\phi)$. Since $\phi$ is not surjective, $V_{i+1}$ is a strict submodule of $V_i$ and $T\subset U\boxtimes V_{i+1}$. Since $V_0=N$ has finite dimension, we cannot indefinitely repeat this construction and decrease the dimension of the submodules $V_i$, hence there must be an integer $n$ such that $U\boxtimes V_n=T$.

We have proved so far that all submodules of $T\subset S$ with simple head are of the form $U\boxtimes V$ for some submodules $U\subset M$ and $V\subset N$. But for each composition factor $L_\alpha$ of $S$ we can find a $T_\alpha$ with $T_\alpha/\mathrm{Rad}(T_\alpha)=L_\alpha$. Then $S=\sum T_\alpha =\sum U_\alpha\boxtimes V_\alpha$ and we are done.
\end{proof}

The submodule lattice of multiplicity free modules can be described in terms of certain oriented diagrams \cite{Alperin}. To be more specific, the diagram $D(M)$ associated to a module $M$ has the composition factors of $M$ as vertices, and there is an edge $L\to L'$ if and only if there is a submodule $U\subset M$ such that $\Head(U)\simeq L$ and $L'$ is a homomorphic image of $\mathrm{Rad}(U)$ (such a module $U$ is unique \cite[Lm 4]{Alperin}). The following proposition describes the diagrams of tensor products $M\boxtimes N$.
\begin{proposition}
The tensor product $M\boxtimes N$ is multiplicity free if and only if both $M$ and $N$ are multiplicity free. If this happens, then the vertices of $D(M\boxtimes N)$ are the tensor products $L_1\boxtimes L_2$ where $L_1$, resp. $L_2$ is a composition factor of $M$, resp. $N$. Moreover, there is an edge $L_1\boxtimes L_2\to L_1'\boxtimes L_2'$ if and only if either $L_1=L_1'$ and there is an edge $L_2\to L'_2$ in $D(N)$ or if $L_2=L_2'$ and there is an edge $L_1\to L'_1$ in $D(M)$.
\end{proposition}
\begin{proof}
We only prove the statement about the edges of $D(M\boxtimes N)$. Let $L_1\boxtimes L_2$ be a composition factor of $M\boxtimes N$ and let $U\subset M$ such that $\Head(U)=L_1$ and $V\subset N$ such that $\Head(V)=L_2$. Then by lemma \ref{lm-head}, $\Head(U\boxtimes V)=L_1\boxtimes L_2$. Thus there is an edge $L_1\boxtimes L_2\to L_1'\boxtimes L_2'$ if and only if $L'_1\boxtimes L'_2$ is a homomorphic image of $\mathrm{Rad}(U)\boxtimes V+U\boxtimes \mathrm{Rad}(V)$, that is if and only if $\hom_{A\boxtimes B}(\mathrm{Rad}(U)\boxtimes V,L'_1\boxtimes L'_2)\ne 0$ or $\hom_{A\boxtimes B}(U\boxtimes \mathrm{Rad}(V),L'_1\boxtimes L'_2)\ne 0$. By the K\"unneth formula, the first condition is equivalent to the fact that $L_1'$ is a homomorphic  image of $\mathrm{Rad}(U)$ and that $L'_2=L_2$ while the second one is equivalent to the fact that $L_2'$ is a homomorphic  image of $\mathrm{Rad}(V)$ and that $L'_1=L_1$.
\end{proof}

\begin{proposition}\label{prop-localizing}
Let $L$ be a simple $A$-module satisfying $\Ext^1_A(L,L)=0$, let $\C$ be a localizing and colocalizing subcategory of $B$-Mod, and let $L\boxtimes \C$ denote the full subcategory of $A\otimes B$-Mod whose objects are isomorphic to tensor products of the form $L\boxtimes M$ where $M$ is an object of $\C$. Then: 
\begin{enumerate}
\item[(i)] $L\boxtimes \C$ is a localizing and colocalizing subcategory of $A\otimes B$-Mod.
\item[(ii)] tensor product by $L$ induces an equivalence of categories $\C\simeq L\boxtimes \C$.
\end{enumerate}
\end{proposition}
\begin{proof}
Statement (ii) follows from the K\"unneth formula and the fact that $\End_A(L)=\kk$. Let us prove (i). The stability of $L\boxtimes \C$ by arbitrary direct sums is obvious, and since $L$ is finite dimensional the canonical morphism $L\boxtimes \prod M_i\to \prod L\boxtimes M_i$ is an isomorphism, which prove the stability by direct products. If $S\subset L\boxtimes N$ then $S=\sum U_\alpha\boxtimes V_\alpha$ by proposition \ref{prop-submodules}. But the only nonzero submodule of $L$ is $L$ itself so that $S=\sum L\boxtimes V_\alpha \simeq L\boxtimes (\sum V_\alpha)$ is an object of $L\boxtimes \C$. The stability by quotients follows from the stability by subobjects. Finally, since $\Ext^1_A(L,L)=0$ and $\End_A(L)=\kk$, the K\"unneth formula shows that $\Ext^1_{A\otimes B}(L\boxtimes N,L\boxtimes N')$ is isomorphic to $\Ext^1_B(N,N')$. Thus, all extensions of $L\boxtimes N$ by $L\boxtimes N'$ are of the form $L\boxtimes E$ where $E$ is an extension of $N$ by $N'$. Hence $L\boxtimes \C$ is stable by extensions.
\end{proof}

\section{On theorems of Steinberg and Clausen-James}\label{app-Stein}
In this appendix, we give new proofs of Steinberg tensor product theorem for $GL_n$ and Clausen and James' theorem, based on theorem \ref{thm-1}.

\begin{lemma}\label{lm-crit}
A strict polynomial functor is simple if and only if it is self-dual and its endomorphism ring has dimension one.
\end{lemma}
\begin{proof}
The condition is necessary by facts \ref{item2} and \ref{item3} from section \ref{subsec-simple}. We prove it is sufficient. Let $L$ be a simple subfunctor of $F$. The composite $F\simeq F^\sharp\twoheadrightarrow L^\sharp\simeq L\hookrightarrow F$ is a nonzero endomorphism of $F$.
Since the endomorphism ring of $F$ has dimension one, this morphism must be a nonzero multiple of the identity, hence an isomorphism. Thus one must have $L=F$.
\end{proof}

\begin{proposition}[Weak Steinberg theorem]\label{prop-weak}
Let $r\ge 0$, let $L_1$ be a left $p^r$-bounded simple functor, and let $L_2$ be any simple functor. Then $L_1\otimes L_2^{(r)}$ is simple. 
\end{proposition}
\begin{proof}
Self-duality of $L_1$, $L_2$ and $I^{(r)}$ and general properties of duality imply that $L_1\otimes L_2^{(r)}$ is self dual. Moreover, since $L_1$ is left $p^r$-bounded, theorem \ref{thm-1} yields an isomorphism:
$$\End_{\PP_\kk}(L_1\otimes L_2^{(r)})\simeq \End_{\PP_\kk}(L_1)\otimes \End_{\PP_\kk}(L_2)\simeq \kk\otimes\kk=\kk\;. $$
Hence $L_1\otimes L_2^{(r)}$ is simple by lemma \ref{lm-crit}.
\end{proof}

Our next task is to prove that the $p$-restricted simple functors are left $p$-bounded. Our proof will use the following proposition, which extends the classification of additive strict polynomial functors proved in \cite{TouzeControl}. 
\begin{proposition}\label{prop-class}
Let $F\in\PP_{d_0,d_1,\dots,d_n,\kk}$ be a strict polynomial functor with $1+n$ variables, such that $F$ is nonzero and additive with respect to each of the last $n$ variables. Let $G$ be the strict polynomial functor defined by $G(V)=F(V,\kk,\dots,\kk)$. Then the $d_i$s, $1\le i\le n$ are powers of $p$, i.e.  $d_i=p^{r_i}$ and there is an isomorphism:
$$F\simeq G\boxtimes I^{(r_1)}\boxtimes\dots\boxtimes I^{(r_n)}\;.$$
\end{proposition}
\begin{proof}
By induction, we can reduce ourselves to proving that $d_n=p^{r_n}$ and that $F$ is isomorphic to $\overline{F}\boxtimes I^{(r_n)}$
where $\overline{F}(V_0,\dots,V_{n-1}):= F(V_0,\dots,V_{n-1},\kk)$. The functors with $n+1$ variables of the form $P\boxtimes \Gamma^\mu$ where $P$ is a projective functor with $n$ variables, homogeneous of multidegree $(d_0,\dots,d_{n-1})$ and $\mu=(\mu_1,\dots,\mu_k)$ is a tuple with $\sum\mu_i=d_n$ form a projective generator of  $\PP_{d_0,d_1,\dots,d_n,\kk}$, thus $F$ is a quotient of a direct sum $\bigoplus P_i\boxtimes \Gamma^{\mu^i}$.

Observe that if $\mu$ has more than one nonzero coefficient, then there are no nonzero morphisms from a functor of the form $P\boxtimes \Gamma^{\mu}$ to $F$. Indeed, for some $n$-tuple $\overline{V}=(V_0,\dots,V_{n-1})$, such a nonzero morphism would induce a nonzero morphism of strict polynomial functors from $P(\overline{V})\otimes\Gamma^{\mu}(-)$ to the additive functor $F(\overline{V},-)$. This would contradict \cite[Thm 2.13]{FS}.

In particular, $F$ is in fact a quotient of $\bigoplus P_i\boxtimes \Gamma^{d_n}= P\boxtimes \Gamma^{d_n}$ with $P=\bigoplus P_i$. And moreover the following composite is zero, in which $\phi=P\boxtimes \mathrm{mult}$ where `$\mathrm{mult}$' refers to the multiplication of the divided power algebra:
$$\bigoplus_{k=1}^{d_n} P\boxtimes (\Gamma^k\otimes \Gamma^{d_n-k})\xrightarrow[]{\phi} P\boxtimes\Gamma^{d_n}\to F \;.$$
Hence $F$ is a quotient of $P\boxtimes (\mathrm{Coker}\,\phi)$. But $\mathrm{Coker}\,\phi$ is nonzero if and only if $d_n=p^{r_n}$ for some $r_n$, and in this case it is equal to $I^{(r_n)}$. Thus $d_n=p^{r_n}$, and we have a surjective map $\psi:P\boxtimes I^{(r_n)}\twoheadrightarrow F$. By replacing the last variable by $\kk$, we obtain a surjective map $\overline{\psi}:P\twoheadrightarrow \overline{F}$. We then take a projective functor with $n$ variables $Q$ and a map $\chi: Q\to P$ whose image is $\mathrm{Ker}\,\overline{\psi}$. Then using additivity with respect with the last variable, one sees that we have a right exact sequence:
$$Q\boxtimes I^{(r_n)}\xrightarrow[]{\chi\boxtimes I^{(r)}} P\boxtimes I^{(r_n)}\xrightarrow[]{\psi} F\to 0\;.$$
This implies that $F$ is isomorphic to $\overline{F}\boxtimes I^{(r_n)}$.
\end{proof}
\begin{corollary}\label{cor-quot}
If $L$ is a simple functor, there exists nonnegative integers $d_0,\dots,d_r$ such that $L$ is a quotient of the functor
$T^{(d_0,\dots,d_r)}= \bigotimes_{0\le i\le r}(\otimes^{d_i})^{(i)}$.
\end{corollary}
\begin{proof}
If $L$ has degree zero, then $L$ is the constant functor $\kk$, hence it is a quotient of $T^{(0)}=\otimes^0=\kk$. Assume that $L$ is not constant. Then there exists a positive integer $n$, the Eilenberg-Mac Lane degree of $L$, such that the functor with $n$ variables 
$$L_{\boxplus_n}: (V_1,\dots,V_n)\mapsto L(V_1\oplus\dots\oplus V_n)$$
contains a nonzero homogeneous direct summand $F$ which is additive with respect to each of its variables (see e.g. \cite[section 2]{TouzeControl} for more details on Eilenberg-Mac Lane degrees for strict polynomial functors). By proposition \ref{prop-class}, $F$ is of the form:
$$F=G\boxtimes I^{(r_1)}\boxtimes\dots\boxtimes I^{(r_n)}$$
where $G$ is a homogeneous functor of degree zero, i.e. a constant functor. In particular, $F$ (hence also $L_{\boxplus_n}$) contains $I^{(r_1)}\boxtimes\dots\boxtimes I^{(r_n)}$ as a direct summand.
Thus we have:
$$0\ne \hom_{\PP_\kk(n)}(I^{(r_1)}\boxtimes\dots\boxtimes I^{(r_n)}, L_{\boxplus_n})\simeq \hom_{\PP_{\kk}}(I^{(r_1)}\otimes\dots\otimes I^{(r_n)},L)\;. $$
Since $L$ is simple, any nonzero morphism with target $L$ is surjective. Thus the inequality above proves that $L$ is a quotient of $I^{(r_1)}\otimes\dots\otimes I^{(r_n)}$. By reordering the factors of this tensor product (and using that $(I^{(k)})^{\otimes d_k}= (\otimes^{d_k})^{(k)}$) we obtain the result.
\end{proof}

We now consider two assertions, indexed by a nonnegative integer $k$.
\begin{enumerate}
\item[$A(k)$:] If $L$ is a $p$-restricted functor of degree $d$ with $d\le k$, then $L$ is a quotient of $\otimes^d$.
\item[$B(k)$:] Let $d$ be a nonnegative integer and let $T$ be a homogeneous functor of positive degree $e$. If $d+pe\le k+1$, then no $p$-restricted simple functor occurs as a composition factor of the tensor product $\otimes^d\otimes T^{(1)}$.
\end{enumerate}

\begin{lemma}\label{lm-r0}
Assertion $A(0)$ is true.
\end{lemma}
\begin{proof}
If $L$ is a simple functor of degree $0$, then $L$ is the constant functor $\kk$, hence it is a quotient of $\otimes^0=\kk$.
\end{proof}

\begin{lemma}\label{lm-r1}
If $A(k)$ is true, then $B(k)$ is true.
\end{lemma}
\begin{proof}
The functor $\otimes^d\otimes T^{(1)}$ admits a filtration whose sucessive quotients are direct sums of functors of the form 
$L_\lambda\otimes T^{(1)}$
where $L_\lambda$ is a simple functor of degree $d$. 
Thus, it suffices to prove that these tensor products $L_\lambda\otimes T^{(1)}$ have no $p$-restricted composition factors. Let us write $\lambda=\alpha+p\beta$, where $\alpha$ is a $p$-restricted partition and $\beta$ is a partition. Since $|\alpha|\le d\le k$, assertion $A(k)$ implies that the simple functor $L_\alpha$ is left $p$-bounded. Thus $L_\lambda\simeq L_{\alpha}\otimes L_\beta^{(1)}$ by the weak Steinberg theorem of proposition \ref{prop-weak}, hence 
$$L_\lambda\otimes T^{(1)}\simeq L_\alpha\otimes \left(L_\beta\otimes T\right)^{(1)}\;.$$
The functor $\left(L_\beta\otimes T\right)^{(1)}$ has composition factors of the form $(L_\gamma)^{(1)}$ with $\gamma\ne (0)$ and since $L_\alpha$ is left $p$-bounded, proposition \ref{prop-weak} implies that the composition factors of $L_\lambda\otimes T^{(1)}$ have the form $L_\alpha\otimes L_\gamma^{(1)}=L_{\alpha+p\gamma}$, hence are not $p$-restricted.
\end{proof}

\begin{lemma}\label{lm-r2}
If $A(k)$ and $B(k)$ are true, then $A(k+1)$ is true.
\end{lemma}
\begin{proof}
Since $A(k)$ is true, it remains to prove that a $p$-restricted functor $L$ of degree $k+1$ is necessarily a quotient of $\otimes^{k+1}$. By corollary \ref{cor-quot}, there exists a tuple of nonnegative integers $(d_0,\dots,d_r)$ such that $L$ is a quotient of a tensor product of the form $T^{(d_0,\dots,d_r)}$.
But assertion $B(k)$ says that such tensor products have no $p$-restricted composition factor except maybe if $(d_0,\dots,d_r)=(k+1,0,\dots,0)$. 
\end{proof}

Lemma \ref{lm-r0}, lemma \ref{lm-r1} and lemma \ref{lm-r2} imply that $A(k)$ is true for all $k\ge 0$. We are now ready to prove:

\begin{theorem}[Steinberg tensor product theorem]\label{thm-Steinberg}
Let $\lambda^0,\dots,\lambda^r$ be $p$-restricted partitions, and let $\lambda=\sum_{i=0}^r p^i\lambda^i$. There is an isomorphism:
$$L_\lambda\simeq L_{\lambda^0}\otimes L_{\lambda^1}^{(1)}\otimes\dots\otimes L_{\lambda^r}^{(r)}\;.$$
\end{theorem}
\begin{proof}
Since $A(k)$ is true for all $k\ge 0$, $p$-restricted simples are quotients of tensor powers $\otimes^d$. Moreover $(\otimes^d)^{(i)}=(I^{(i)})^{\otimes d}$ is a quotient of $(\Gamma^{p^i})^{\otimes d}$. Thus for all $k\le r$,
$\bigotimes_{i< k}L_{\lambda^i}^{(i)}$ is left $p^{k}$-bounded. An induction on $k$ using proposition \ref{prop-weak} shows that each tensor product $\bigotimes_{i\le k}L_{\lambda^i}^{(i)}$ is simple.
\end{proof}
\begin{remark}\label{rk-St-unstable}
If $\lambda=(\lambda_1,\dots,\lambda_k)$, then $L_\lambda(\kk^n)$ is a simple polynomial $GL_n(\kk)$-module if $n\ge k$ and is zero if $n<k$. Thus, theorem \ref{thm-Steinberg} actually implies the Steinberg tensor product theorem for polynomial representations of $GL_n(\kk)$, for all values of $n$ (and in particular without requiring that the representations are stable). Finally, all simple rational representations of $GL_n(\kk)$ can be obtained by tensoring simple polynomial representations of $GL_n(\kk)$ by a power of the determinant representation. Thus, theorem \ref{thm-Steinberg} implies the classical Steinberg tensor product theorem as in \cite[II.3.17]{Jantzen}.
\end{remark}

\begin{theorem}[Clausen and James' theorem]
A simple functor $L$ is $p$-restricted if and only if $\hom_{\PP_\kk}(L,\otimes^d)=\hom_{\PP_\kk}(\otimes^d,L)$ is nonzero.
\end{theorem}
\begin{proof}
Property $A(k)$ gives the `only if' part. Conversely, assume that the highest weight $\lambda$ of $L$ is not $p$-restricted. Using euclidean division, we can write $\lambda=\lambda^0+p\lambda^1$ with $\lambda^0$ $p$-restricted and $\lambda^1$ nonzero. Thus $L\simeq L_{\lambda^0}\otimes L_{\lambda^1}$ by Steinberg tensor product theorem. By property $A(k)$, $L_{\lambda^0}$ is left and right $p$-bounded, so that by theorem \ref{thm-1}  $\hom_{\PP}(L,\otimes^d)=\hom_{\PP}(\otimes^d,L)=0$.
\end{proof}

\begin{remark}\label{rk-Kuhn}
There already exists a functorial proof of Steinberg tensor product theorem in the literature \cite[Thm 7.11]{Kuhn}. However, the proof given in this appendix is quite different from that in \cite{Kuhn}. Let us stress two differences. First the proof in \cite{Kuhn} uses finite fields, while the size of the ground field plays no role in our proof. Second, to obtain a concrete form of \cite[Thm 7.11]{Kuhn}, one needs to know the classification of simple representations of symmetric groups. On the contrary, our proof does not use any knowledge of representations of symmetric groups. Better still, our reasonning also proves Clausen and James' theorem, so we can actually use our approach to derive the classification of simple representations of symmetric groups from the classification of simple representations of $GL_n$.
\end{remark}

Steinberg tensor product theorem tells us that if $L_\lambda$ is simple and $p$-restricted and $L_\mu$ is simple, then $L_\lambda\otimes L_\mu^{(1)}$ is simple. The following statement completes the picture regarding tensor products of simple objects. 

\begin{theorem}\label{thm-tensprespres}
Let $L$ and $L'$ be both simple and $p$-restricted. Then $L\otimes L'$ is not simple, unless one of the two is the constant functor $\kk$.
\end{theorem}

The remainder of the section is devoted to the proof of theorem \ref{thm-tensprespres}.

\begin{lemma}\label{lm-crs}
Let $d$ be a positive integer, and let $L$ be a simple quotient of $\otimes^d$. The following injection induced by the tensor product is not surjective:
$$ \hom_{\PP_\kk}(\otimes^{d},L)\otimes\hom_{\PP_\kk}(I,I)\hookrightarrow \hom_{\PP_\kk}(\otimes^{d+1},L\otimes I)\;.$$
\end{lemma}
\begin{proof}
Fix a vector space $V$ equipped with an isomorphism $\kk^d\oplus \kk\simeq V$. Let $\iota_1:\kk\hookrightarrow V$, $\iota_2:\kk^{d}\hookrightarrow V$, $\pi_1:V\to \kk$ and $\pi_2:V\to \kk^d$ be the associated canonical maps. Since $\End_{\PP_\kk}(I)\simeq\kk$, any nonzero map $\phi$ in the image of the injection of lemma \ref{lm-crs} is of the form $\phi=f\otimes\Id$ for a nonzero $f$. Thus the following composite is nonzero (it equals the map induced by $f$):
$$(\kk^{d})^{\otimes d}\otimes\kk\xrightarrow[]{(\iota_2)^{\otimes d}\otimes \iota_1}
V^{\otimes d+1}\xrightarrow[]{\phi} L(V)\otimes V\xrightarrow[]{L(\pi_2)\otimes\pi_1}L(\kk^{d})\otimes \kk\;.\quad(*)$$
For all morphisms $f:\otimes^d\to L$, we define a morphism $\psi_f:\otimes^{d+1}\to L\otimes I$ by
$\psi_f(x_1\otimes\dots\otimes x_{d+1})= f(x_2\otimes \dots\otimes x_{d+1})\otimes x_1.$
If $f$ is nonzero,  then $\psi_f$ is nonzero, while for $\phi=\psi_f$ the composite $(*)$ is zero. In particular $\psi_f$ is not in the image of the inclusion.
\end{proof}

\begin{lemma}\label{lm-nonzero}
Let $d$ be a positive integer, and let $L$ be a simple quotient of $\otimes^d$. Let $L^{(d-1,1)}$ be the homogeneous summand of bidegree $(d-1,1)$ of the bifunctor $(V,W)\mapsto L(V\oplus W)$. There is an isomorphism $L^{(d-1,1)}\simeq F_L\boxtimes I$ where $F_L$ is a \emph{nonzero} homogeneous functor of degree $d-1$.
\end{lemma}

\begin{proof} Proposition \ref{prop-class} provides an isomorphism $L^{(d-1,1)}\simeq F_L\boxtimes I$. We have to prove that $F_L$ is nonzero.
By using the sum-diagonal adjunction and the K\"unneth formula, we obtain that $\hom_{\PP_\kk}(\otimes^{d+1},L\otimes I)$ is isomorphic to
$$\hom_{\PP_\kk}(\otimes^d,L)\otimes\End_{\PP_\kk}(I)\;\oplus\; \hom_{\PP_\kk}(\otimes^d,F_L\otimes I)\otimes \End_{\PP_\kk}(I)\;.$$
For dimension reasons, lemma \ref{lm-crs} implies that $\hom_{\PP_\kk}(\otimes^d,F_L\otimes I)$ is nonzero. Hence $F_L$ is nonzero.
\end{proof}

\begin{proof}[Proof of theorem \ref{thm-tensprespres}]
We will show that the dimension of $\End_{\PP_{\kk}}(L\otimes L')$ is not one. To this purpose, we use the sum-diagonal adjunction and the K\"unneth formula. We obtain that the vector space $\End_{\PP_{\kk}}(L\otimes L')$ contains 
$$\End_{\PP_{\kk}}(L)\otimes \End_{\PP_{\kk}}(L')\;\oplus\;\hom_{\PP_\kk}(L,F_L\otimes I)\otimes \hom_{\PP_\kk}(L',F_{L'}\otimes I) $$
as a direct summand, with $F_L$ and $F_{L'}$ defined as in lemma \ref{lm-nonzero}. By using the sum-diagonal adjunction and the K\"unneth formula again, we obtain that $\hom_{\PP_\kk}(L,F_L\otimes I)$ contains 
$\End_{\PP_\kk}(F_L)\otimes \End_{\PP_\kk}(I)$
as a direct summand (and similarly for $L'$). But lemma \ref{lm-nonzero} asserts that $F_L$ and $F_{L'}$ are nonzero, so that the dimension of the corresponding endomorphism spaces is at least one. Thus, the dimension of $\End_{\PP_{\kk}}(L\otimes L')$ is at least two.
\end{proof}


\begin{thebibliography}{99}

\bibitem{Alperin} J.L. Alperin, Diagrams for modules. J. Pure Appl. Algebra 16 (1980), no. 2, 111--119.

\bibitem{ABW} K.~Akin, D.~Buchsbaum, J.~Weyman, Schur functors and Schur complexes.  Adv. in Math.  44  (1982), no. 3, 207--278.

\bibitem{AndersenExt} H.H.~Andersen, Extensions of modules for algebraic groups. Amer. J. Math. 106 (1984), no. 2, 489--504.

\bibitem{Andersen} H.H.~Andersen, $p$-filtrations and the Steinberg module, J. Algebra 244 (2001), 664--683.

\bibitem{Axtell} J.~Axtell, Spin  polynomial  functors  and  representations  of  Schur  superalgebras,  Represent.  Theory 17 (2013), 584--609.

\bibitem{Benson} D. Benson, Representations and cohomology. I. Basic representation theory of finite groups and associative algebras. Second edition. Cambridge Studies in Advanced Mathematics, 30. Cambridge University Press, Cambridge, 1998.

\bibitem{BC} D. Benson,  J. Conway, Diagrams for modular lattices. J. Pure Appl. Algebra 37 (1985), no. 2, 111--116. 

\bibitem{BK} C. Bessenrodt, A. Kleshchev, On tensor products of modular representations of symmetric groups, Bulletin of the London Mathematical Society 32 (2000), 292--296.

\bibitem{BrunK} J. Brundan, J. Kujawa, A new proof of the Mullineux conjecture, J. Algebraic Combin. 18 (2003), 13--39.

\bibitem{BMT} L.~Breen, R.~Mikhailov, A. Touz\'e, Derived functors of the divided powers, Derived functors of the divided power functors, to appear in Geometry and Topology.

\bibitem{Chalupnik2} M.~Cha{\l}upnik, Koszul duality and extensions of exponential functors.  Adv. Math.  218  (2008),  no. 3, 969--982.

\bibitem{Chalupnik15} M.~Cha{\l}upnik, Derived Kan Extension for strict polynomial functors, Int. Math. Res. Not. (2015),  doi:10.1093/imrn/rnu269. 

\bibitem{Clausen} M. Clausen, Letter place algebras and a characteristic-free approach to the representation theory of the general linear and symmetric groups. I,II. Adv. in Math. 33 (1979), 161--191 and Adv. in Math. 38 (1980), 152--177.

\bibitem{CR}  C. Curtis, I. Reiner, Methods of representation theory. Vol. I. With applications to finite groups and orders. Reprint of the 1981 original. Wiley Classics Library. John Wiley \& Sons, Inc., New York, 1990. xxiv+819 pp. ISBN: 0-471-52367-4

\bibitem{DEN} S. Doty, K. Erdmann, D. Nakano, Extensions of modules over Schur algebras, symmetric groups and Hecke algebras, Algebr. Represent. Theory (2004), no. 1, 67--100.

\bibitem{Djament} A.~Djament, Sur l'homologie des groupes unitaires \`a coefficients polynomiaux. (French) [Homology of unitary groups with polynomial coefficients]  J. K-Theory 10 (2012), no. 1, pp 87--139. 

\bibitem{DV} A.~Djament, C.~Vespa, Sur l'homologie des groupes orthogonaux et symplectiques \`a coefficients tordus. (French) [Homology of orthogonal and symplectic groups with twisted coefficients]  Ann. Sci. Éc. Norm. Supér. (4)  43  (2010),  no. 3, 395--459.

\bibitem{DonkinExt} S.~Donkin, On Ext$^1$ for semisimple groups and infinitesimal subgroups. Math. Proc. Cambridge Philos. Soc. 92 (1982), no. 2, 231--238.

\bibitem{Drupieski} C.~Drupieski, Cohomological finite-generation for finite supergroup schemes
Adv. Math. 288 (2016), 1360--1432. 

\bibitem{FFSS} V.~Franjou, E.~Friedlander, A.~Scorichenko, A.~Suslin, General linear and functor cohomology over finite fields,   Ann. of Math. (2)  150  (1999),  no. 2, 663--728.

\bibitem{FK} B. Ford, A. Kleshchev, A proof of the Mullineux Conjecture, Math. Z. 226(1997), 2, 267--308.

\bibitem{FS} E.~Friedlander, A.~Suslin, Cohomology of finite group schemes over a field,
 Invent. Math. 127 (1997), 209--270.
 
\bibitem{Green} J.A.~Green, Polynomial representations of ${\rm GL}_{n}$. Second corrected and augmented edition. Lecture Notes in Mathematics, 830. Springer, Berlin, 2007. 

\bibitem{HY} J.~Hong, O.~Yacobi, Quantum polynomial functors, {\tt arXiv:1504.01171}.

\bibitem{James} G. James, The decomposition of tensors over fields of prime characteristic. Math. Z. 172 (1980), no. 2, 161--178.

\bibitem{Jantzen} J.C.~Jantzen, Representations of algebraic groups. Second edition. Mathematical Surveys and Monographs, 107. American Mathematical Society, Providence, RI, 2003. 

\bibitem{Kuhn} N.~Kuhn, A  stratification  of  generic  representation  theory  and  generalized  Schur  algebras,
K- theory J. 26 (2002), 15--49.

\bibitem{KN}
A. Kleshchev and D. Nakano, On comparing the cohomology of general linear and symmetric groups, Pacific J. Math. 201 (2001), 339--355. 

\bibitem{KS}
A.S. Kleshchev and J. Sheth, On extensions of simple modules over symmetric and algebraic
groups, J. Algebra 221 (1999), 705--722. 


\bibitem{Krause} H. Krause, Koszul, Ringel and Serre Duality for strict polynomial functors, Compos. Math. 149 (2013), 996-1018.

\bibitem{Krause2} H. Krause, The highest weight structure for strict polynomial functors, {\tt arXiv:1405.1691}.

\bibitem{MacDonald} I.G.~Macdonald, Symmetric functions and Hall polynomials.
Second edition. Oxford Mathematical Monographs. Oxford Science Publications.
The Clarendon Press, Oxford University Press, New York, 1995. 

\bibitem{Martin} S.~Martin, Schur algebras and representation theory. Cambridge
Tracts in Mathematics, 112. Cambridge University Press, Cambridge, 1993.

\bibitem{Panorama} Rational representations, the Steenrod algebra and functor homology,  27--53, Panor. Synth\`eses, 16, Soc. Math. France, Paris, 2003.

\bibitem{Reischuk} R. Reischuk, The adjoints of the Schur functor, {\tt arXiv:1601.03513}

\bibitem{Ringel} C.M.~Ringel, The category of modules with good filtrations over a quasi-hereditary algebra has almost split sequences.  Math. Z.  208  (1991),  no. 2, 209--223.

\bibitem{SFB} A.~Suslin, E.~Friedlander, C.~Bendel, Infinitesimal $1$-parameter subgroups and cohomology.  J. Amer. Math. Soc.  10  (1997),  no. 3, 693--728.

\bibitem{Totaro} B.~Totaro, Projective resolutions of representations of GL(n), J. Reine Angew. Math. 482 (1997) 1--13.

\bibitem{TouzeClassical}A. Touz\'e, Cohomology of classical algebraic groups from the functorial viewpoint, Adv. in Math. 225 (2010), no. 1, 33--68.

\bibitem{TouzeTroesch} A. Touz\'e,Troesch complexes and extensions of strict polynomial functors.
Ann. Sci. \'Ec. Norm. Sup\'er. (4), 45(1):53--99, 2012.

\bibitem{TouzeRingel} A.~Touz\'e, Ringel duality and derivatives of non-additive functors, J. Pure Appl. Algebra 217 (2013), no. 9, 1642--1673.

\bibitem{TouzeClasses} A.~Touz\'e, A construction of the universal classes for algebraic groups with the twisting spectral sequence, Transform. Groups 18 (2013), no. 2, 539--556.

\bibitem{TouzeEML} A.~Touz\'e, Bar complexes and extensions of classical exponential functors, Ann. Inst. Fourier (Grenoble) 64 (2014), no. 6, 2563--2637.

\bibitem{TouzeControl} A.~Touz\'e, A functorial control of integral torsion in homology, to appear in Fund. Math., {\tt arXiv:1310.2877}

\bibitem{TouzeICRA} A.~Touz\'e, Computations and applications of some homological constants for polynomial representations of GLn, preprint.
\end{thebibliography}
\end{document}